\documentclass[12pt,reqno]{amsart}
\usepackage[T1]{fontenc}
\usepackage{lmodern,microtype,amsmath,amssymb,amsthm,mathtools,mathrsfs}
\usepackage[a4paper,hmargin=26mm,vmargin=25mm,headheight=13pt,headsep=7mm]{geometry}
\usepackage{booktabs,array,xurl,needspace}
\usepackage[colorlinks=true,linkcolor=blue,citecolor=blue,urlcolor=blue]{hyperref}

\usepackage{xcolor}
\definecolor{RevisionBlue}{rgb}{0,0,1}
\newif\ifBlueChanges
\BlueChangestrue
\DeclareRobustCommand{\BlueChange}[1]{%
  \ifBlueChanges{\color{RevisionBlue}#1}\else#1\fi}

\hypersetup{hidelinks}
\numberwithin{equation}{section}
\newtheorem{theorem}{Theorem}[section]
\newtheorem{proposition}[theorem]{Proposition}
\newtheorem{lemma}[theorem]{Lemma}
\newtheorem{corollary}[theorem]{Corollary}
\theoremstyle{definition}
\newtheorem{definition}[theorem]{Definition}
\theoremstyle{remark}
\newtheorem{remark}[theorem]{Remark}
\newtheorem{example}[theorem]{Example}
\newcommand{\R}{\mathbb R}
\newcommand{\Z}{\mathbb Z}
\newcommand{\N}{\mathbb N}
\newcommand{\Dt}{\mathsf D}
\newcommand{\lD}{\underline{\mathsf D}}
\newcommand{\uD}{\overline{\mathsf D}}
\newcommand{\Dpm}{\mathsf D}
\newcommand{\lDim}{\underline{\operatorname{Dim}}}
\newcommand{\uDim}{\overline{\operatorname{Dim}}}

\newcommand{\Ct}{\mathscr C}
\newcommand{\Ht}{\mathcal H}
\newcommand{\Mt}{\mathcal M}
\DeclareMathOperator\hDim{\mathrm{Dim}_{\mathrm{H}}}
\newcommand{\UMdim}{\dim_{\mathrm{UM}}}
\newcommand{\LMdim}{\dim_{\mathrm{LM}}}

\newcommand{\Pf}{\mathcal P_f}
\newcommand{\diam}{\operatorname{diam}}

\newcommand{\ind}{\mathbf 1}
\newcommand{\ee}{\mathbf e_1}

\allowdisplaybreaks[1]
\hypersetup{pdftitle={Projection theorems for macroscopic intermediate dimensions},pdfauthor={Jinjun Li, Lifeng Xi},pdfsubject={Macroscopic intermediate dimensions, capacity profiles, and Grassmannian projections}}

\title[Macroscopic intermediate dimensions and projections]{Projection theorems for macroscopic intermediate dimensions}

\author[J. Li]{Jinjun Li}
\address{School of Mathematics and Information Science, Guangzhou University, Guangzhou 510006, P. R. China}
\email{lijinjun@gzhu.edu.cn}

\author[L. Xi]{Lifeng Xi$^\ast$}
\address{School of Mathematics and Statistics, Ningbo University, Ningbo 315211, P. R. China}
\email{xilifeng@nbu.edu.cn}

\thanks{$^\ast$Corresponding author. Email: xilifengningbo@yahoo.com; xilifeng@nbu.edu.cn}

\date{\BlueChange{27} September 2026}

\begin{document}
	
\begin{abstract}
We introduce lower and upper macroscopic intermediate dimensions of unbounded subsets of Euclidean space by combining covering contributions from all dyadic annuli beyond a given scale. For every unbounded Borel set $E\subset\R^d$ and $m\in \{1,\ldots,d-1\}$, we prove that the macroscopic intermediate dimensions of $P_VE$ equal the corresponding intrinsic capacity profiles for almost every 
$m$-dimensional linear subspaces $V\subset \R^d$. One full-measure set works for both dimensions and every parameter $\theta\in[0,1]$. The proof combines a covering--capacity comparison with an exponent gap, fractional-moment estimates for projected coverings, and capacity-weighted retention of source shells. The dimensions recover macroscopic Hausdorff dimension at parameter zero, are invariant under quasi-isometries, and are continuous at positive parameters. The upper macroscopic intermediate dimension agrees with the  upper discrete intermediate dimension introduced in \cite{LXZ}, whereas the lower one can differ. We compute the projected dimensions of bounded-base discrete digit Moran sets, including the self-similar case obtained when the bases and digit sets are constant, and give applications to Lipschitz graphs and arithmetic sumsets.
\end{abstract}

\maketitle
\enlargethispage{3pt}

\section{Introduction and main results}\label{sec:intro}

Projection theorems relate the dimension of a set to the dimensions of its images in lower-dimensional spaces. Marstrand's theorem and Kaufman's energy proof \cite{Marstrand,Kaufman} give the classical planar Hausdorff result. For box and packing dimensions, capacity profiles describe  typical projection values \cite{FH96,FH97,Howroyd}. Falconer, Fraser and Kempton \cite{FFK} introduced intermediate dimensions by restricting covering sizes, and Burrell, Falconer and Fraser \cite{BFF} proved the corresponding projection formulae. Related preservation results use local covering information \cite{FFS}; see also the survey \cite{Fraser25}.

For unbounded sets, Barlow and Taylor \cite{BT89,BT92} introduced macroscopic Hausdorff dimension. This dimension disregards geometry below a fixed unit scale and measures the distribution of a set at infinity. Daw and Seuret \cite{DS} developed a macroscopic potential method and proved a planar projection theorem for macroscopic Hausdorff dimension. 
%Their energy theorem applies in every Euclidean dimension. We use that energy theorem, but prove the higher-dimensional projection statement here rather than assuming an extension of their planar result.
Counting and mass dimensions provide another large-scale approach to study unbounded sets and the Marstrand-type theorems are also establised for them, see \cite{LM,Glasscock}.

Li, Xi and Zhang \cite{LXZ} defined discrete intermediate dimensions using scale-dependent finite ranges of dyadic annuli and an auxiliary-parameter supremum. A projection creates an additional difficulty: points from arbitrarily distant source annuli may project into the same target annulus. We therefore include all source annuli beyond the outer scale, since even very distant annuli may contribute to the same projected annulus. We call the resulting quantities \emph{macroscopic intermediate dimensions}. We use the term macroscopic intermediate dimension to distinguish this full-tail definition from the discrete intermediate dimension of \cite{LXZ}.

The change affects the lower spectrum. The upper macroscopic intermediate dimension agrees with the  upper discrete intermediate dimension in \cite{LXZ}, while the new lower dimension can be strictly larger than the lower discrete intermediate dimension.  Proposition~\ref{prop:old-comparison} and Example~\ref{ex:separating-lower} state these comparisons precisely. The main contribution is the full-tail approach provides a unified framework in which both the lower and upper macroscopic intermediate dimensions satisfy exact projection formulae.

For an integer $n\ge0$, define the dyadic shells of $\R^d$ by
\[
 S_0=B(0,1),\qquad S_n=B(0,2^n)\setminus B(0,2^{n-1})\quad(n\ge1).
\]
Throughout the paper, all balls are open. For $E\subset\R^d$, an integer $N\ge1$, $s\ge0$, and
$0\le\theta\le1$, define $\mathcal{H}_N^{s,\theta}(E)=\sum_{n=N}^\infty  \mathcal{H}_{n,N}^{s,\theta}(E)$, where
\[
 \begin{aligned}
 \mathcal{H}_{n,N}^{s,\theta}(E)&=
 \inf\left\{\sum_i\left(\frac{r_i}{2^n}\right)^s:
 \begin{array}{l}E\cap S_n\subset\bigcup_iB(x_i,r_i),\\
 x_i\in S_n,\quad 1\le r_i\le2^{n-N\theta}
 \end{array}\right\}.
 \end{aligned}
\]
We define the \emph{lower} and \emph{upper macroscopic intermediate dimensions} by
\begin{align}
 \Dt^-_\theta(E)&=\inf\left\{s\ge0:
 \liminf_{N\to\infty}\frac{\log\mathcal{H}_N^{s,\theta}(E)}{N\log2}\le0\right\},
 \label{eq:def-minus}\\
 \Dt^+_\theta(E)&=\inf\left\{s\ge0:
 \limsup_{N\to\infty}\frac{\log\mathcal{H}_N^{s,\theta}(E)}{N\log2}\le0\right\}.
 \label{eq:def-plus}
\end{align}
The content of an empty shell is zero, and the full-tail sum is allowed to be infinite. Every shell has a finite unit-ball cover, so its local content is finite. Covers may be taken finite: for fixed $n,N,s$, each admissible ball contributes at least $2^{-ns}$, and an infinite cover has infinite cost. We use $\log0=-\infty$ and $\log(+\infty)=+\infty$.

%Section~\ref{sec:basic} proves the endpoint identities, monotonicity, continuity at positive parameters, and quasi-isometry invariance. 

To state the projection theorem, we introduce lower and upper macroscopic intermediate dimension profiles, which in motivated by the work of  \cite{BFF}. 

For $R\ge1$, an integer $1\le m\le d$, and $0<s<m$, define
\begin{equation}
 \begin{split}
 K_R^{s,m}(z)&=
 \begin{cases}
 1,&|z|\le1,\\
 |z|^{-s},&1<|z|\le R,\\
 R^{m-s}|z|^{-m},&|z|>R,
 \end{cases}
  \end{split}
 \label{eq:kernel}
\end{equation}
 
We define the capacity of $A$ by
\[
 C_R^{s,m}(A)=
 \left(\inf_{\mu\in\Pf(A)}
 \iint K_R^{s,m}(x-y)\,d\mu(x)\,d\mu(y)\right)^{-1},
\]
where $\Pf(A)$ consists of finitely supported probability measures on $A$;
$C_R^{s,m}(\varnothing)=0$. The integral is the $K_R^{s,m}$-energy of $\mu$. For bounded nonempty $A$, the kernel is bounded below by a positive constant on $A-A$, so $0<C_R^{s,m}(A)<\infty$. For an arbitrary set $E\subset\R^d$, define
\[
 \Ct_N^{s,\theta,m}(E)=
 \sum_{n=N}^\infty 2^{-ns}C_{2^{n-N\theta}}^{s,m}(E\cap S_n),\quad N\ge 1.
\]
The \emph{lower} and \emph{upper macroscopic intermediate dimension profiles} are, respectively, defined by
\[
\begin{aligned}
 \lD_\theta^m(E)
 &=\inf\left(\left\{s\in(0,m):
 \liminf_{N\to\infty}\frac{\log\Ct_N^{s,\theta,m}(E)}{N\log2}
 \le0\right\}\cup\{m\}\right),\\
 \uD_\theta^m(E)
 &=\inf\left(\left\{s\in(0,m):
 \limsup_{N\to\infty}\frac{\log\Ct_N^{s,\theta,m}(E)}{N\log2}
 \le0\right\}\cup\{m\}\right).
\end{aligned}
\]
Thus $0\le\lD_\theta^m(E)\le\uD_\theta^m(E)\le m$. These quantities are defined from the shell pieces of $E$ without using any projection. Bounded sets have both profiles zero.
When the two cases are treated simultaneously, we use the shorthand
$\Dpm_\theta^{-,m}(E)=\lD_\theta^m(E)$ and
$\Dpm_\theta^{+,m}(E)=\uD_\theta^m(E)$.

Fix integers $1\le m<d$. Let $G(d,m)$ be the Grassmannian of
$m$-dimensional linear subspaces of $\R^d$, equipped with its
rotation-invariant probability measure $\gamma_{d,m}$, see \BlueChange{\cite{Mattila}}. For $V\in G(d,m)$,
let $P_V:\R^d\to V$ be orthogonal projection. We regard $V$ as a Euclidean
space with its induced metric. In the case that $d=2$ and $m=1$, we also use $P_\alpha(x_1,x_2)=x_1\cos\alpha+x_2\sin\alpha$, $0\le\alpha<\pi$,
and  $\gamma_{2,1}$ corresponds to normalized Lebesgue measure in $\alpha$. 

The target set $P_VE$ is covered using its own dyadic shells in $V$; source shell labels are not retained. The following is the main result.
\begin{theorem}\label{thm:projection}
Let $E\subset\R^d$ be an unbounded Borel set and $m\in \{1,\ldots,d-1\}$.
There is a set $\mathcal G_{E,m}\subset G(d,m)$ of full
$\gamma_{d,m}$-measure such that
\begin{equation}
 \Dt^-_\theta(P_VE)=\lD_\theta^m(E),\qquad
 \Dt^+_\theta(P_VE)=\uD_\theta^m(E)
 \label{eq:projection}
\end{equation}
for every $V\in\mathcal G_{E,m}$ and every $\theta\in[0,1]$. The two equalities are separate: the lower and upper profile values need not coincide.
\end{theorem}

The macroscopic Hausdorff dimension can be written as
\begin{equation}
	\hDim E=\inf\left\{s\ge0:\sum_{n\ge1}\mathcal{H}_n^s(E)<\infty\right\},
	\label{eq:macro}
\end{equation}
where  $\mathcal{H}_n^s(E)=\mathcal{H}_{n,n}^{s,0}(E)$, see 
\cite{BT92,DS}.
 Propositions~\ref{prop:basic-endpoints} and~\ref{prop:recovery} give
\[
	\Dt^-_0(E)=\Dt^+_0(E)=\hDim E,\quad
	\lD_0^m(E)=\uD_0^m(E)=\min\{\hDim E,m\}.
	\]
For $(d,m)=(2,1)$ and $\theta=0$, Theorem~\ref{thm:projection} recovers the following theorem of Daw and Seuret \cite[Theorem~2.8]{DS}: 
\begin{corollary}\label{rem:DS-endpoint}For every Borel set $E\subset \R^2$,
	\[
	\hDim(P_\alpha E)=\min\{\hDim E,1\}
	\]
for Lebesgue almost every $\alpha\in[0,\pi)$.
\end{corollary}

Theorem~\ref{thm:projection} identifies the two typical projection dimensions with their respective intrinsic capacity profiles. Its proof has two parts. For the upper estimate, a mixed covering functional measures each large source ball in terms of the number and size of balls needed to cover its projection. A gap between two test exponents makes the covering--capacity loss depend on the outer scale rather than on a possibly much larger source shell index. A fractional-moment estimate then controls the complete target tail. For the lower estimate, source shells are retained according to their capacity weights. Neither argument bounds the number of contributing source shells in terms of the outer scale.

%Discrete Moran sets furnish explicit applications \cite{DYX,LCX,YDX}. We treat a bounded-base vector-digit model, including stationary self-similar digit sets, and prove its macroscopic intermediate dimension formula directly. Further applications concern subsets of Lipschitz graphs and arithmetic sumsets. These applications follow from the general projection theorem and do not replace it as the main result.

Throughout this paper, the notation $f\lesssim g$ indicates that there
exists a constant $C>0$ such that $f\le Cg$, and $f\gtrsim g$ is defined
analogously. We write $f\asymp g$ if $f\gtrsim g$ and $f\lesssim g$.

The rest of the paper is organized as follows. Section~\ref{sec:basic} proves the basic properties of macroscopic intermediate dimensions. Sections~\ref{sec:capacity} and~\ref{sec:projection}
prove the capacity and projection results.
Section~\ref{sec:applications} gives exact cases, including discrete Moran sets, and applications.

\section{Basic properties of macroscopic intermediate dimensions}\label{sec:basic}
We begin with equivalent characterizations of the macroscopic
intermediate dimensions and use them to establish their basic properties.
\subsection{Equivalent definitions and range }
\begin{lemma}\label{lem:zero-limit}
For $0<\theta\le 1$ and $E\subset \R^d$,
\[
 \Dt^-_\theta(E)=\inf\left\{s\ge0:
 \liminf_{N\to\infty}\mathcal H_N^{s,\theta}(E)=0\right\}
\]
and
\[
 \Dt^+_\theta(E)=\inf\left\{s\ge0:
 \limsup_{N\to\infty}\mathcal H_N^{s,\theta}(E)=0\right\}.
\]
\end{lemma}
\begin{proof}
We only establish the first equality, as the second one can be verified in a similar way.
For $s>t\ge0$, it follows from the definition that
$\mathcal H_{n,N}^{s,\theta}(E)
\le2^{-N\theta(s-t)}\mathcal H_{n,N}^{t,\theta}(E)$. Hence
\begin{equation}
 \mathcal H_N^{s,\theta}(E)
 \le2^{-N\theta(s-t)}\mathcal H_N^{t,\theta}(E).
 \label{eq:strict-cover}
\end{equation} 
Write $\alpha^-=\inf\left\{s\ge0:
\liminf_{N\to\infty}\mathcal H_N^{s,\theta}(E)=0\right\}.$
If $s>\alpha^-$, choose $t$ with $\alpha^-<t<s$ such that
$\liminf_{N\to\infty}\mathcal H_N^{t,\theta}(E)=0$.  It follows from
\eqref{eq:strict-cover} that there is an
increasing sequence $N_k\to\infty$ for which
$\mathcal H_{N_k}^{s,\theta}(E)\to0$. Consequently
\[
 \liminf_{N\to\infty}
 \frac{\log\mathcal H_N^{s,\theta}(E)}{N\log2}\le0,
\]
and therefore $\Dt^-_\theta(E)\le s$. Letting $s\downarrow\alpha^-$ gives
$\Dt^-_\theta(E)\le\alpha^-$.

On the other hand, let $s>t>\Dt^-_\theta(E)$. Then
\[
\liminf_{N\to\infty}\frac{\log\mathcal{H}_N^{t,\theta}(E)}{N\log2}\le 0.
\]
Choose
$\varepsilon<\theta(s-t)$. There exists an increasing sequence $N_k\to \infty$ such that 
$\mathcal {H}_{N_k}^{t,\theta}(E)\le2^{\varepsilon N_k}$, so it follows from \eqref{eq:strict-cover} that
$\mathcal H_{N_k}^{s,\theta}(E)\to0$ and hence
$\liminf_{N\to\infty}\mathcal H_N^{s,\theta}(E)=0$. Thus
$\alpha^-\le s$. Letting first $s\downarrow t$ and then
$t\downarrow\Dt^-_\theta(E)$ gives $\alpha^-\le\Dt^-_\theta(E)$.
\end{proof}

\begin{proposition}\label{prop:single}
	For every $E\subset\R^d$ and $0\le\theta\le1$,
	\begin{equation}
		\Dt^+_\theta(E)=\inf\left\{s\ge0:
		\sum_{n\ge1}\mathcal H_{n,n}^{s,\theta}(E)<\infty\right\}.
		\label{eq:single-series}
	\end{equation}
	For $\theta>0$, we also have
	\begin{equation}
		\Dt^+_\theta(E)=\inf\left\{s\ge0:\limsup_{n\to\infty}
		\frac{\log\mathcal H_{n,n}^{s,\theta}(E)}{n\log2}\le0 \right\}.
		\label{eq:tddy}
	\end{equation}
\end{proposition}
\begin{proof}
For $n\ge N$, $2^{(1-\theta)n}\le 2^{n-N\theta}.$
Hence every cover admissible for $\mathcal H_{n,n}^{s,\theta}(E)$ is also admissible for $\mathcal H_{n,N}^{s,\theta}(E)$, and therefore 
	\begin{equation}
		\mathcal{H}_{N,N}^{s,\theta}(E)\le \mathcal{H}_N^{s,\theta}(E)=\sum_{n=N}^\infty \mathcal{H}_{n,N}^{s,\theta}(E)\le\sum_{n= N}^\infty\mathcal H_{n,n}^{s,\theta}(E).
		\label{eq:single-tail-comparison}
	\end{equation}

Let $A$ denote the right-hand side of
	\eqref{eq:single-series}.
	Suppose first that $\sum_{n=1}^\infty \mathcal H_{n,n}^{s,\theta}(E)<\infty$. Its tails tend to zero, so
	\eqref{eq:single-tail-comparison} gives
	\[
	\limsup_{N\to\infty}
	\frac{\log\mathcal H_N^{s,\theta}(E)}{N\log2}\le0.
	\]
	This implies that $\Dt^+_\theta(E)\le A$.
	Conversely,  let $t>\Dt^+_\theta(E)$. Then the last limsup, with $t$ in
	place of $s$, is nonpositive. The left inequality in
	\eqref{eq:single-tail-comparison} gives
\begin{equation}\label{inequality}
	\limsup_{n\to\infty}\frac{\log\mathcal H_{n,n}^{t,\theta}(E)}{n\log2}\le0.
\end{equation}
	If $s>t$ and $\theta>0$, every ball occurring in the definition of
	$\mathcal H_{n,n}^{t,\theta}(E)$ satisfies $r_i/2^n\le2^{-n\theta}$, whence
\begin{equation}\label{eq:single-exponent-gap}
		\mathcal H_{n,n}^{s,\theta}(E)\le2^{-n\theta(s-t)}\mathcal H_{n,n}^{t,\theta}(E).
	\end{equation}
	Choose $0<\varepsilon<\theta(s-t)$. \eqref{inequality} implies
	$\mathcal H_{n,n}^{t,\theta}(E)\le2^{\varepsilon n}$ for all sufficiently large $n$, and
	the inequality \eqref{eq:single-exponent-gap} then shows that
	$\sum_{n=1}^\infty\mathcal H_{n,n}^{s,\theta}(E)<\infty$. Hence
	$A\le s$ for every $s>t$. Letting first $s\downarrow t$ and then
	$t\downarrow\Dt^+_\theta(E)$ we obtain the reverse inequality $A\le \Dt^+_\theta(E)$.
	This proves \eqref{eq:single-series} for $\theta>0$. When $\theta=0$,
	both sides of \eqref{eq:single-series} equal the critical exponent of
	the series in \eqref{eq:macro}, namely $\hDim E$.
	
	Let $B$ denote the right-hand side of
	\eqref{eq:tddy}. Since $\mathcal H_{N,N}^{s,\theta}(E)\le\mathcal H_N^{s,\theta}(E)$, it follows from \eqref{eq:def-plus} that
	$B\le\Dt^+_\theta(E)$. Conversely, let $t>B$. The  inequality \eqref{eq:single-exponent-gap} implies convergence of
	$\sum_n\mathcal H_{n,n}^{s,\theta}(E)$ for every $s>t$, so \eqref{eq:single-series} gives
	$\Dt^+_\theta(E)\le t$ and hence $\Dt^+_\theta(E)\le B$. 
\end{proof}

\begin{proposition}\label{prop:basic-endpoints}
For every $E\subset\R^d$ and $0\le\theta\le1$,
\begin{equation}
 \Dt^-_0(E)=\Dt^+_0(E)=\hDim E,
 \qquad
 \hDim E\le\Dt^-_\theta(E)\le\Dt^+_\theta(E)\le d.
 \label{eq:basic-endpoints}
\end{equation}
Bounded sets have both dimensions zero. The zero-limit characterization
in Lemma~\ref{lem:zero-limit} also holds at $\theta=0$.
\end{proposition}
\begin{proof}
At $\theta=0$ the admissible radii in
$\mathcal H_{n,N}^{s,0}(E)$ are precisely $1\le r_i\le2^n$; hence
$
 \mathcal H_N^{s,0}(E)=\sum_{n\ge N}\mathcal H_n^s(E).
$
Let $s>\hDim E$. It follows from  \eqref{eq:macro} that $\sum_{n=1}^\infty \mathcal{H}_n^s(E)<\infty$ and hence $\mathcal H_N^{s,0}(E)\to 0$ as $N\to \infty$. This proves that $\Dt^-_0(E)\le\Dt^+_0(E)\le\hDim E$.  Conversely, if $s<\hDim E$, then the nonnegative series
$\sum_n\mathcal H_n^s(E)$ diverges, so every tail
$\mathcal H_N^{s,0}(E)$ is infinite. Hence $s$ is inadmissible for both
critical exponents. Therefore
$\Dt^-_0(E)=\Dt^+_0(E)=\hDim E$, and the zero-limit characterization
in Lemma~\ref{lem:zero-limit} also holds at $\theta=0$.

For $0\le\theta\le1$ and $n\ge N$, every admissible
$(N,\theta)$-cover is also admissible in the definition of
$\mathcal H_n^s(E)$, since $2^{n-N\theta}\le2^n$. Therefore
$
 \mathcal H_N^{s,\theta}(E)
 \ge \sum_{n\ge N}\mathcal H_n^s(E).
$
If $s<\hDim E$, the right-hand side is infinite for every $N$, so $s$
is inadmissible in either \eqref{eq:def-minus} or \eqref{eq:def-plus}.
Thus $\hDim E\le\Dt^-_\theta(E)$. 

We next prove $\Dt^+_\theta(E)\le d$. A maximal
$\tfrac12$-separated subset of $E\cap S_n$ has at most $C_d2^{nd}$
points, by comparison of the disjoint balls of radius $1/4$ about its
points with a ball of radius $2^n+1$. The corresponding unit balls,
whose centres lie in $S_n$, cover $E\cap S_n$. They are admissible for
all $n\ge N$, because $2^{n-N\theta}\ge1$. Consequently, for $s>d$,
\[
 \mathcal H_N^{s,\theta}(E)
 \le C_d\sum_{n\ge N}2^{-n(s-d)}
 \le C_{d,s}2^{-N(s-d)}.
\]
Both normalized logarithmic limits are therefore negative, and letting
$s\downarrow d$ gives $\Dt^+_\theta(E)\le d$. Finally, if $E$ is
bounded, then $E\cap S_n=\varnothing$ for all sufficiently large $n$;
hence the tail contents eventually vanish and both dimensions are
zero.
\end{proof}

\subsection{Mass dimensions and the upper endpoint}\label{sec:mass}
For a locally finite set $E\subset\R^d$, its upper and lower mass
dimensions are, respectively,
\[
\begin{aligned}
 \UMdim E&=\limsup_{R\to\infty}
 \frac{\log\max\{1,\#(E\cap B(0,R))\}}{\log R},\\
 \LMdim E&=\liminf_{R\to\infty}
 \frac{\log\max\{1,\#(E\cap B(0,R))\}}{\log R}.
\end{aligned}
\]
Let $\mathcal N_r(A)$ be the least number of radius-$r$ balls covering
$A$, with arbitrary centres. Write $M_n(E)$ for the least number of
unit balls centred in $S_n$ covering $E\cap S_n$.

The following result shows that the upper macroscopic intermediate
dimension coincides with the upper mass dimension for uniformly discrete
sets. Recall that a set $E\subset\R^d$ is called \emph{uniformly discrete}
if there exists $c>0$ such that $|x-y|\ge c$ for $x,y\in E$ with $x\ne y$.

\begin{proposition}\label{prop:mass-endpoint}
If $E\subset\R^d$ is uniformly discrete, then
$
 \Dt^+_1(E)=\UMdim E.
$
\end{proposition}
\begin{proof}
At $\theta=1$, the series in Proposition~\ref{prop:single} is
$\sum_n M_n(E)2^{-ns}$, because the only admissible radius in
$\mathcal H_{n,n}^{s,1}(E)$ is one. Its critical exponent is
\[
 L:=\limsup_{n\to\infty}
 \frac{\log\max\{1,M_n(E)\}}{n\log2}.
\]
To see this, if $s>L$, then for some $\varepsilon>0$ and all large $n$,
$M_n(E)\le2^{n(s-\varepsilon)}$, so the series converges. If $s<L$,
then $M_n(E)2^{-ns}\ge2^{\varepsilon n}$ along an infinite sequence
for some $\varepsilon>0$, and its terms do not tend to zero. Hence
Proposition~\ref{prop:single} yields $\Dt^+_1(E)=L$.

Let $a_n=\#(E\cap S_n)$. If $c>0$ is a separation constant for \BlueChange{$E$},
volume comparison of the disjoint $c/2$-balls centred at points of
$E\cap B(x,1)$ gives a constant $Q=Q(c,d)$ such that every unit ball
contains at most $Q$ points of $E$. Therefore
$
 M_n(E)\le a_n\le Q M_n(E).
$
The first inequality uses unit balls centred at the points of
$E\cap S_n$; the second applies the preceding bound to a minimal unit
cover. Thus $L$ is also the upper exponential growth rate of $a_n$.

Put $A_N=\#(E\cap B(0,2^N))=\sum_{n=0}^N a_n$. Since $a_N\le A_N$,
the upper growth rate of $a_n$ does not exceed that of $A_N$. Conversely,
if $\varepsilon>0$, then eventually
$a_n\le2^{n(L+\varepsilon)}$; summing the resulting geometric bound
(and absorbing the finitely many initial shells) gives
$A_N\le C_\varepsilon2^{N(L+\varepsilon)}$. Hence
\[
 L=\limsup_{N\to\infty}
 \frac{\log\max\{1,A_N\}}{N\log2}.
\]
For $2^{N-1}\le R<2^N$, monotonicity of the counting function places
$\#(E\cap B(0,R))$ between the corresponding dyadic counts. It follows
that the last limsup is exactly $\UMdim E$, completing the proof.
\end{proof}

\begin{remark}\label{rem:mass-cover-comparison}
The covering argument in the preceding proof applies to arbitrary
$E\subset\R^d$. For a locally finite set it yields
\begin{equation}
 \Dt^+_1(E)=\limsup_{R\to\infty}
 \frac{\log\max\{1,\mathcal N_1(E\cap B(0,R))\}}{\log R}
 \le\UMdim E.
 \label{eq:mass-cover-comparison}
\end{equation}
Similarly,
\[
 \liminf_{R\to\infty}
 \frac{\log\max\{1,\mathcal N_1(E\cap B(0,R))\}}{\log R}
 \le\LMdim E,
\]
with equality when $E$ is uniformly discrete. Both inequalities follow
from $\mathcal N_1(A)\le\#A$ and can be strict without uniform
discreteness. For example, let
\[
 F=\bigcup_{n\ge1}
 \{3\cdot2^{n-2}+j2^{-3n}:0\le j<2^n\}\subset\R.
\]
Each shell contains $2^n$ points in one unit ball, so
$\UMdim F=\LMdim F=1$, whereas both unit-cover exponents and
$\Dt^-_1(F),\Dt^+_1(F)$ are zero. We remark that even for uniformly discrete sets,
$\Dt^-_1(E)=\LMdim E$ can fail, see Example \ref{ex:separating-lower}.
\end{remark}

\subsection{Monotonicity and continuity}
Analogously to the discrete intermediate dimensions introduced in
\cite{LXZ}, the macroscopic intermediate dimensions, as functions of
$\theta$, are nondecreasing on $[0,1]$ and continuous on $(0,1]$.

\begin{proposition}\label{prop:dimension-continuity}
Both spectra are nondecreasing on $[0,1]$. For either sign and
$0<\theta<\varphi\le1$,
\begin{equation}
 0\le\Dt^\pm_\varphi(E)-\Dt^\pm_\theta(E)
 \le\left(1-\frac{\theta}{\varphi}\right)
 \bigl(d-\Dt^\pm_\theta(E)\bigr).
 \label{eq:dimension-continuity}
\end{equation}
In particular, both spectra are continuous on $(0,1]$.
\end{proposition}
\begin{proof}
If $0\le\theta\le\varphi\le1$, every $(N,\varphi)$-admissible cover is
also $(N,\theta)$-admissible. Hence
$\mathcal H_N^{s,\theta}(E)\le\mathcal H_N^{s,\varphi}(E)$, and both
critical exponents are nondecreasing.

Fix $0<\theta<\varphi\le1$, $0\le s<d$, and put
$t=d-(\theta/\varphi)(d-s)$. We prove the content comparison
\begin{equation}
 \mathcal H_N^{t,\varphi}(E)\le C_d\mathcal H_N^{s,\theta}(E).
 \label{eq:continuity-content}
\end{equation}
For $n\ge N$, take a cover admissible for
$\mathcal H_{n,N}^{s,\theta}(E)$ and set $u=2^{n-N\varphi}$. Retain each ball of radius $r\le u$. Since $t\ge s$ and $r/2^n\le1$,
$
\left(\frac r{2^n}\right)^t
\le
\left(\frac r{2^n}\right)^s.
$ If $r>u$, choose a maximal
$u/2$-separated subset of $B(x,r)\cap E\cap S_n$. Its $u$-balls cover
that set, their centres belong to $S_n$, and volume comparison shows
that their number is at most $C_d(r/u)^d$. The corresponding contribution to the \(t\)-covering sum is at most
$
 C_d(r/2^n)^s(r/2^n)^{d-s}(u/2^n)^{t-d}
 \le C_d(r/2^n)^s,
$
because $r/2^n\le2^{-N\theta}$,
$u/2^n=2^{-N\varphi}$, and
$\theta(d-s)+\varphi(t-d)=0$. Taking infima, summing over $n\ge N$,
and then approximating the infima if necessary proves
\eqref{eq:continuity-content}.

Let $D=\Dt^\pm_\theta(E)<d$. By \eqref{eq:strict-cover},
if an exponent satisfies the corresponding defining condition,
then every larger exponent also satisfies it.
Since $D$ is the infimum of the exponents satisfying this condition,
every $s$ with $D<s<d$ satisfies it. After taking the relevant lower or upper limit
in \eqref{eq:continuity-content}, $t$ is admissible at $\varphi$.
Consequently
$
 \Dt^\pm_\varphi(E)
 \le d-\frac{\theta}{\varphi}(d-s).
$
Letting $s\downarrow D$ proves \eqref{eq:dimension-continuity}. If
$D=d$, that estimate follows from Proposition~\ref{prop:basic-endpoints}.
Right continuity at every $\theta>0$ follows by letting
$\varphi\downarrow\theta$ in the estimate, and left continuity follows
by letting $\theta\uparrow\varphi$. This includes left continuity at
$1$.
\end{proof}

The following example shows that there may be a jump at the macroscopic endpoint.
\begin{example}\label{ex:zero-jump}
Consider the subset of a line in $\R^2$
\[
 F=\bigcup_{n\ge n_0}
 \{(3\cdot2^{n-2}+j,0): |j|\le L_n,\ j\in\Z\},
 \qquad L_n=\lfloor2^{n-\sqrt n}\rfloor,
\]
where $n_0$ is large enough that every displayed block is contained in
$S_n$.  Indeed, its centre is $3\cdot2^{n-2}$ and
$L_n=o(2^n)$.

At $\theta=0$, the $n$-th block is covered by one ball of radius
$L_n+1$. Hence, for every $s>0$,
\[
 \mathcal H_n^s(F)\le
 C_s\left(\frac{L_n+1}{2^n}\right)^s
 \le C_s2^{-s\sqrt n}.
\]
Since $\sum_n2^{-s\sqrt n}<\infty$, Proposition~\ref{prop:basic-endpoints}
gives
\[
 \Dt^-_0(F)=\Dt^+_0(F)=0.
\]

Now fix $\theta>0$ and $0<s<1$.  At the outer scale $N$, the admissible
radii in the shell $S_N$ satisfy
$r_i\le R_N:=2^{(1-\theta)N}$.  The $N$-th block contains
$2L_N+1\asymp2^{N-\sqrt N}$ integer-spaced points.  A ball of radius
$r_i\ge1$ contains at most $2r_i+1\le3r_i$ of these points, so every
admissible cover satisfies
$
 \sum_i r_i\ge cL_N.
$
Because $s<1$ and $r_i\le R_N$, we also have
$r_i^s\ge r_iR_N^{s-1}$.  Therefore
\[
 \mathcal H_N^{s,\theta}(F)
 \ge \mathcal H_{N,N}^{s,\theta}(F)\ge c_s2^{-Ns}R_N^{s-1}L_N
 \ge c_s2^{\theta(1-s)N-\sqrt N}.
\]
The normalized lower logarithmic rate is thus positive, and
$\Dt^-_\theta(F)\ge1$.  On the other hand, every shell of the line can
be covered by $O(2^n)$ unit balls, so Proposition~\ref{prop:single}
(or directly the geometric series estimate) gives
$\Dt^+_\theta(F)\le1$. Consequently
\[
 \Dt^-_\theta(F)=\Dt^+_\theta(F)=1\qquad(\theta>0).
\]
\end{example}

The preceding example shows that continuity at $\theta=0$ may fail in general. The following proposition gives a simple quantitative condition ensuring continuity at \(\theta=0\).
\begin{proposition}\label{prop:zero-cont}
Let $h=\hDim E<d$. Suppose that for every $q>h$ there are
$\varepsilon_q>0$ and $C_q<\infty$ such that
\begin{equation}
 A_N^q(E):=\sum_{n\ge N}\mathcal H_n^q(E)\le C_q2^{-\varepsilon_q N}
 \quad\text{for all large }N.
 \label{eq:macro-decay}
\end{equation}
Then $
\lim_{\theta\downarrow0}\Dt_\theta^\pm(E)
=\operatorname{Dim}_{\mathrm H}E.
$ It suffices to check
\eqref{eq:macro-decay} for exponents decreasing to $h$.
\end{proposition}
\begin{proof}
Fix $h<q<s<d$, $n\ge N$, and put $R=2^{n-N\theta}$. Start with a
cover admissible for $\mathcal H_n^q(E)$. A ball of radius $r\le R$ is retained. Since $s>q $ and $r/2^n\le1$,
$
\left(\frac r{2^n}\right)^s
\le
\left(\frac r{2^n}\right)^q.$ If $r>R$, a maximal $R/2$-separated subset of
$B(x,r)\cap E\cap S_n$ supplies at most $C_d(r/R)^d$ radius-$R$ balls
with centres in $S_n$. The corresponding contribution to the \(s\)-covering sum is at most
$
 C_d(r/R)^d(R/2^n)^s
 =C_d(r/2^n)^d2^{N\theta(d-s)}
 \le C_d(r/2^n)^q2^{N\theta(d-s)}.
$
Taking infima over the original covers and summing over $n\ge N$ gives
\begin{equation}
 \mathcal H_N^{s,\theta}(E)
 \le C_{d,q,s}2^{N\theta(d-s)}A_N^q(E).
 \label{eq:zero-cont-content}
\end{equation}
By \eqref{eq:macro-decay}, if
$0\le\theta<\varepsilon_q/(d-s)$, then
$
 \mathcal H_N^{s,\theta}(E)
 \le C2^{-N(\varepsilon_q-\theta(d-s))}.
$
Thus $\Dt^+_\theta(E)\le s$. Proposition~\ref{prop:basic-endpoints}
now gives
$
 h\le\Dt^-_\theta(E)\le\Dt^+_\theta(E)\le s
$
for all sufficiently small $\theta$. Given $\eta>0$, choose
$h<q<s<\min\{h+\eta,d\}$; the preceding inequalities prove that both
dimensions tend to $h$ as $\theta\downarrow0$. Since their values at
zero equal $h$, both spectra are continuous there. The same argument
only requires a sequence $q_j\downarrow h$ for which
\eqref{eq:macro-decay} holds: given $\eta$, choose $q_j<h+\eta$ and
then choose $s$ between $q_j$ and $\min\{h+\eta,d\}$.
\end{proof}

\subsection{Quasi-isometries}\label{sec:qi}
A map $f:E\to F$ is a quasi-isometry if, for some $L\ge1$ and $c\ge0$,
\[
 L^{-1}|x-x'|-c\le|f(x)-f(x')|\le L|x-x'|+c
\]
for all $x,x'\in E$, and every point of $F$ is within distance $c$
of $f(E)$. The ambient dimensions of $E$ and $F$ may differ.  Translations, nonzero dilations, and bi-Lipschitz maps are special cases. 

The following proposition, whose idea of proof comes from \cite{LXZ}, shows that the lower and upper macroscopic intermediate dimensions are invariant under quasi-isometries. 

\begin{proposition}\label{prop:qi}
If $E\subset\R^d$ and $F\subset\R^k$ are quasi-isometric, then
$\Dt^\pm_\theta(E)=\Dt^\pm_\theta(F)$ for all $\theta\in[0,1]$.
In particular, both spectra are independent of the origin and of the
fixed unit scale used consistently in the definition.
\end{proposition}
\begin{proof}
The empty-set case is immediate. Quasi-isometric spaces are bounded simultaneously, so the bounded case follows from Proposition~\ref{prop:basic-endpoints}. We may therefore suppose that $E$ and $F$ are nonempty and unbounded.
Choose for each $y\in F$ a point $p(y)\in E$ with
$|y-f(p(y))|\le c+1$. Fix a point of $E$.
The distortion inequalities, applied relative to this point, imply
that for all large shells
\[
 y\in S_j\ \Longrightarrow\ p(y)\in\bigcup_{|n-j|\le K}S_n
\]
for a fixed integer $K$. If several preimages belong to one source
ball of radius $r\ge1$, their target points have diameter at most
$Ar$, where $A$ is independent of the shell.

Use the source outer scale $N-K$ for the target scale $N$.
For $|n-j|\le K$, the ratio between the maximal admissible radii at the source and target scales is
$
 \frac{2^{n-(N-K)\theta}}{2^{j-N\theta}}
 =2^{n-j+K\theta}\le2^{2K}.
$
Thus a target piece of diameter at most $Ar$ can be covered by a
bounded number of legal target balls, with centres in $S_j$.
If its diameter is below one, a unit ball is sufficient. The normalized
cost is at most $C_s(r/2^n)^s$. A source shell contributes to at most
$2K+1$ target shells. Taking infima and summing gives
\begin{equation}
 \mathcal H_N^{s,\theta}(F)\le C_s\mathcal H_{N-K}^{s,\theta}(E)
 \quad\text{for all large }N.
 \label{eq:qi-content}
\end{equation}
The chosen map $p$ is a coarse inverse: the lower distortion inequality
bounds $|p(f(x))-x|$ uniformly. It satisfies analogous distortion bounds
and gives the reverse content inequality. Fixed index shifts and fixed
multiplicative constants do not change either critical exponent, proving the result. Translating the origin changes only the shell-localization constant. For a fixed unit scale $a>0$, use shells with radii $a2^n$, allowed radii $a\le r_i\le a2^{n-N\theta}$, and normalization $a2^n$. Dilation by $a^{-1}$ reduces this convention to the one used above. This proves the final assertion.
\end{proof}

\subsection{Comparison with discrete intermediate dimensions}\label{sec:old}
Li, Xi and Zhang \cite{LXZ} introduced lower and upper discrete
intermediate dimensions for unbounded sets in $\R^d$. Their definitions
have the following equivalent dyadic form. For $\eta>0$ and $0\le \theta\le 1$, put
\[
 J_{N,\eta}=\{N,\ldots,\lceil N/\eta\rceil\},
 \qquad J_{N,0}=\{n:n\ge N\},
\]
and
\[
 \Phi_{2^{-N}}^{s,\theta}(F)=\sup_{0\le\eta\le\theta}
 \sum_{n\in J_{N,\eta}}\mathcal H_{n,N}^{s,\eta}(F).
\]
Then
\[
\begin{aligned}
 \lDim_\theta F
 &=\inf\left\{s\ge0:\liminf_{N\to\infty}
 \frac{\log\Phi_{2^{-N}}^{s,\theta}(F)}{N\log2}\le0\right\},\\
 \uDim_\theta F
 &=\inf\left\{s\ge0:\limsup_{N\to\infty}
 \frac{\log\Phi_{2^{-N}}^{s,\theta}(F)}{N\log2}\le0\right\}.
\end{aligned}
\]
To compare with the definitions for all $\delta\downarrow0$, let
$2^{-N}\le\delta<2^{-N+1}$. The ranges of dyadic annuli involved in the definition remain the same, while the corresponding maximal admissible radii satisfy $2^{n-N\eta}\le2^n\delta^\eta\le2\,2^{n-N\eta}$.
Subdivision and recentering by a bounded number of balls therefore give
$
 \Phi_\delta^{s,\theta}(F)\le\Phi_{2^{-N}}^{s,\theta}(F)
 \le C_{d,s}\Phi_\delta^{s,\theta}(F).
$
Since $-\log\delta=N\log2+O(1)$, both critical exponents agree.
Thus these are the existing discrete intermediate dimensions, with no
additional dimension notation. When indexing by the dyadic level, we
also write $\Phi_N^{s,\theta}(F)=\Phi_{2^{-N}}^{s,\theta}(F)$.

\begin{proposition}\label{prop:old-comparison}
For every $F\subset\R^d$ and $0\le\theta\le1$,
\begin{equation}
 \lDim_\theta F\le\Dt^-_\theta(F)
 \le\Dt^+_\theta(F)=\uDim_\theta F.
 \label{eq:old-comparison}
\end{equation}
\end{proposition}
\begin{proof}
For $0\le\eta\le\theta$, the upper radius bound at $\eta$ is no
smaller than the one at $\theta$, and $J_{N,\eta}\subset\{n:n\ge N\}$.
Moreover, $N\in J_{N,\theta}$. It follows directly from the definitions
that
\[
 \mathcal H_{N,N}^{s,\theta}(F)
 \le\Phi_{2^{-N}}^{s,\theta}(F)\le\mathcal H_N^{s,\theta}(F).
\]
Indeed, the first inequality uses the term $n=N$ with $\eta=\theta$;
for the second, for each fixed $\eta\le\theta$,
\[
 \sum_{n\in J_{N,\eta}}\mathcal H_{n,N}^{s,\eta}(F)
 \le\sum_{n\in J_{N,\eta}}\mathcal H_{n,N}^{s,\theta}(F)
 \le\mathcal H_N^{s,\theta}(F),
\]
and one then takes the supremum over $\eta$. Passing to the lower and
upper normalized logarithmic limits gives
$\lDim_\theta F\le\Dt^-_\theta(F)$ and
$\uDim_\theta F\le\Dt^+_\theta(F)$.

Assume $\theta>0$. If $t>\uDim_\theta F$, exponent monotonicity makes
$t$ admissible in the definition of $\uDim_\theta F$. The first
inequality above then yields
\[
 \limsup_{N\to\infty}
 \frac{\log\mathcal H_{N,N}^{t,\theta}(F)}{N\log2}\le0.
\]
The second characterization in Proposition~\ref{prop:single} implies
$\Dt^+_\theta(F)\le t$. Letting $t\downarrow\uDim_\theta F$ proves the
reverse upper inequality. If $\theta=0$, then $J_{N,0}=\{n:n\ge N\}$
and the only auxiliary parameter is $\eta=0$, so
$\Phi_{2^{-N}}^{s,0}(F)=\mathcal H_N^{s,0}(F)$; equality is immediate\BlueChange{.}
\end{proof}

The folllwoing example shows that  the new lower spectrum need not agree with either lower mass dimension or the  lower discrete intermediate dimension.
\begin{example}\label{ex:separating-lower}
Put $q=5/4$, $n_k=\lfloor q^k\rfloor$, and, after discarding finitely many initial levels, let
\[
 A=\bigcup_{k\ge k_0}
 \{(3\cdot2^{n_k-2}+j)\ee: j\in\Z,\ |j|\le L_k\},
 \quad L_k=\lfloor2^{3n_k/4}\rfloor.
\]
Each block lies in $S_{n_k}$. At parameter one,
\begin{equation}
 \LMdim A=\frac35,
 \qquad \lDim_1 A=\frac{39}{64},
 \qquad \Dt^-_1(A)=\frac{11}{16},
 \qquad \Dt^+_1(A)=\frac34.
 \label{eq:separating-values}
\end{equation}
\end{example}
\begin{proof}
	Write $\delta=1/4$ and
	$
	A_k=\{(3\cdot 2^{n_k-2}+j)e_1:j\in\mathbb Z,\ |j|\le L_k\}.
	$
	After discarding finitely many initial terms, $(n_k)$ is strictly
	increasing and
$
	L_k\asymp 2^{(1-\delta)n_k}=2^{3n_k/4}=o(2^{n_k}).
$
	Hence
	\[
	2^{n_k-1}
	<
	3\cdot2^{n_k-2}-L_k
	<
	3\cdot2^{n_k-2}+L_k
	<
	2^{n_k},
	\]
	so $A_k\subset S_{n_k}$, and all other sufficiently large annuli are
	empty. Moreover,
$
	\#A_k=2L_k+1\asymp 2^{(1-\delta)n_k}.
$
	
	For $N\ge1$, put
$
	a_N=\min\{n_k:n_k\ge N\}.
$
	Fix $0<s<1$, let $n=n_k\ge N$, and let $0\le\eta\le1$. Set
$
	R=2^{n-N\eta}.
$
	If $\{B(x_i,r_i)\}_i$ is an admissible cover of $A_k$, then each ball
	contains at most $2r_i+1\le3r_i$ points. Thus
$
	\sum_i r_i\gtrsim 2^{(1-\delta)n}.
$
	Since $r_i\le R$ and $s<1$,
$
	r_i^s\ge r_iR^{s-1},
$
	and therefore
	\[
	\mathcal H^{s,\eta}_{n,N}(A)
	\gtrsim
	2^{-ns}R^{s-1}2^{(1-\delta)n}
	=
	2^{-\delta n+(1-s)N\eta}.
	\]
	
Conversely, if $L_k+1\le R$, one admissible ball covers $A_k$ and gives
	a contribution $O(2^{-s\delta n})$. If $L_k+1>R$, subdividing the block
	at scale $R$, with centres chosen in $A_k$, gives
	\[
	\mathcal H^{s,\eta}_{n,N}(A)
	\lesssim
	\frac{L_k}{R}\left(\frac{R}{2^n}\right)^s
	\lesssim
	2^{-\delta n+(1-s)N\eta}.
	\]
	Consequently,
	\begin{equation}
		c_s2^{-\delta n+(1-s)N\eta}
		\le
		\mathcal H^{s,\eta}_{n,N}(A)
		\le
		C_s\left(
		2^{-s\delta n}
		+
		2^{-\delta n+(1-s)N\eta}
		\right).
		\tag{2.16}
	\end{equation}
	
	We first compute the full-tail dimensions at $\theta=1$. Summing
	(2.16) over the occupied annuli and using the geometric growth of
	$n_k$ gives
	\[
	c_s2^{-\delta a_N+(1-s)N}
	\le
	\mathcal H^{s,1}_N(A)
	\le
	C_s\left(
	2^{-s\delta a_N}
	+
	2^{-\delta a_N+(1-s)N}
	\right).
	\tag{2.17}
	\]
	Furthermore,
	\[
	\liminf_{N\to\infty}\frac{a_N}{N}=1,
	\qquad
	\limsup_{N\to\infty}\frac{a_N}{N}=q.
	\tag{2.18}
	\]
	Indeed, the first value is attained along $N=n_k$, whereas for
	$N=n_k+1$ we have $a_N=n_{k+1}$ and $a_N/N\to q$. Since the first term
	on the right-hand side of (2.17) decays exponentially for every
	$s>0$, the critical behaviour is determined by
	$-\delta a_N+(1-s)N$. Hence
	\[
	D^-_1(A)=1-q\delta=\frac{11}{16},
	\qquad
	D^+_1(A)=1-\delta=\frac34.
	\tag{2.19}
	\]
		
For the  lower discrete intermediate dimension, choose
	$\eta=N/a_N$. Then $a_N\in J_{N,\eta}$, and (2.16) gives
$
	\Phi^{s,1}_N(A)
	\gtrsim
	2^{-\delta a_N+(1-s)N^2/a_N}.
$
	Conversely, if $n>N$ belongs to $J_{N,\eta}$, then
	$\eta\le N/(n-1)$. Put
	$
	\omega_N(N)=1 $ and 
	$\omega_N(n)=\min\left\{1,\frac{N}{n-1}\right\}$ for $n>N.$
	The function
$
	-\delta n+(1-s)N\omega_N(n)
$
	decreases by at least $\delta$ as $n$ increases. Hence the sum over
	the occupied annuli is controlled by its first term, and
	\[
	c_s2^{-\delta a_N+(1-s)N^2/a_N}
	\le
	\Phi^{s,1}_N(A)
	\le
	C_s\left(
	2^{-s\delta a_N}
	+
	2^{-\delta a_N+(1-s)N^2/a_N}
	\right).
	\tag{2.20}
	\]
	Writing $r_N=a_N/N$, the relevant normalized exponent is
$	-\delta r_N+\frac{1-s}{r_N}.
$
	This is decreasing in $r_N$, so its lower limit is obtained along
	$r_N\to q$. The critical equation
$
	-\delta q+\frac{1-s}{q}=0
$
	therefore gives
	\[
\lDim_1 A
	=
	1-\delta q^2
	=
	1-\frac14\left(\frac54\right)^2
	=
	\frac{39}{64}.
	\tag{2.21}
	\]
	
	Finally, since $n_k$ grows geometrically we have
$
	\#A_k\asymp 2^{(1-\delta)n_k}$ and 
$
	\sum_{j\le k}\#A_j\asymp 2^{(1-\delta)n_k}.
$
 Between the $k$-th and $(k+1)$-st
	blocks, the counting function is therefore comparable to
	$2^{(1-\delta)n_k}$, while the radius can be comparable to
	$2^{n_{k+1}}$. It follows that
	\[
	\dim_{\mathrm{LM}}A
	=
	(1-\delta)\lim_{k\to\infty}\frac{n_k}{n_{k+1}}
	=
	\frac{1-\delta}{q}
	=
	\frac35.
	\]

\end{proof}

\section{Capacities and macroscopic intermediate dimension profiles}\label{sec:capacity}

We establish the basic properties of the dimension profiles and
compare capacities with covering sums to prepare for the proof
of the projection theorem.
\subsection{Kernel comparisons}
For $R'\ge R\ge1$ and $0<u<v<m$, we have
\begin{equation}
 K_R^{v,m}(z)\le K_{R'}^{v,m}(z)\le\left(\frac{R'}{R}\right)^{m-v}K_R^{v,m}(z),
 \qquad K_R^{u,m}(z)\le R^{v-u}K_R^{v,m}(z).
 \label{eq:kernel-comparisons}
\end{equation}
In fact, write \(t=|z|\). If \(R'\geq R\), the definition gives

$$ \frac{K_{R'}^{v,m}(z)}{K_R^{v,m}(z)} = \begin{cases} 1,&t\leq R,\\ (t/R)^{m-v},&R<t\leq R',\\ (R'/R)^{m-v},&t>R'. \end{cases} $$
Since \(m-v>0\), this ratio lies between \(1\) and \((R'/R)^{m-v}\). Hence
$$ K_R^{v,m}(z)\leq K_{R'}^{v,m}(z) \leq (R'/R)^{m-v}K_R^{v,m}(z). $$
Likewise, for \(0<u<v<m\),
$$ \frac{K_R^{u,m}(z)}{K_R^{v,m}(z)} = \begin{cases} 1,&t\leq1,\\ t^{v-u},&1<t\leq R,\\ R^{v-u},&t>R. \end{cases} $$
Each expression is at most \(R^{v-u}\), proving
$ K_R^{u,m}(z)\leq R^{v-u}K_R^{v,m}(z). $
The second comparison in \eqref{eq:kernel-comparisons}  gives
\begin{equation}
 \Ct_N^{v,\theta,m}(E)
 \le 2^{-N\theta(v-u)}\Ct_N^{u,\theta,m}(E).
 \label{eq:strict-capacity}
\end{equation}

%This holds also for infinite sums. Thus the defining exponent sets for the profiles are upward closed. Below a profile the corresponding logarithmic growth rate is positive, possibly infinite; above it the rate is nonpositive. For $\theta>0$, a further increase of the exponent provides a strictly negative growth margin at the relevant scales. No continuity or zero of a growth-rate function is assumed.

We include the following approximation lemma because the capacities are defined using finitely supported probability measures, whereas the Radon measures arising at the macroscopic endpoint ($\theta=0$) need not be finitely supported.

\begin{lemma}\label{lem:finite-energy}
Let $A\subset\R^d$ be bounded, nonempty and Borel, $R\ge1$, and $0<s<m$.  If $\nu$ is a
finite Borel measure concentrated on $A$, then
\begin{equation}
 \iint K_R^{s,m}(x-y)\,d\nu(x)d\nu(y)
 \ge \frac{\nu(A)^2}{C_R^{s,m}(A)}.
 \label{eq:finite-energy-capacity}
\end{equation}
Moreover, $A\subset B$ implies
$C_R^{s,m}(A)\le C_R^{s,m}(B)$.
\end{lemma}
\begin{proof}

We may assume that $\nu(A)>0$ since the case $\nu(A)=0$ is immediate, and normalize $\nu$ to a probability measure
$
\sigma=\frac{\nu}{\nu(A)}.
$
Choose a cube containing $A$, and for each $j\ge1$ partition it into finitely many half-open cubes $Q$ of side length $\delta_j$, where $\delta_j\downarrow0$. For every $Q$ with $A\cap Q\neq\varnothing$, choose a point $x_Q\in A\cap Q$, and define
$
\sigma_j=\sum_Q \sigma(A\cap Q)\delta_{x_Q}.
$
Then $\sigma_j$ is a finitely supported probability measure on $A$.

Let $T_j:A\to A$ send each $x\in A\cap Q$ to $x_Q$. Since
$
|T_jx-x|\le \sqrt d\delta_j,
$
we have
$
|(T_jx-T_jy)-(x-y)|
\le 2\sqrt d\delta_j
$
for all $x,y\in A$.   Since \(A\) is bounded and \(K_R^{s,m}\) is continuous, the function \((x,y)\mapsto K_R^{s,m}(x-y)\) is uniformly continuous on \(\overline A\times\overline A\).   Hence
\[
\sup_{x,y\in A}
\left|
K_R^{s,m}(T_jx-T_jy)-K_R^{s,m}(x-y)
\right|
\longrightarrow0.
\]

As $\sigma_j=\sigma \circ T_j^{-1}$, it follows that

\[
\iint K_R^{s,m}(x-y)\,d\sigma_j(x)\,d\sigma_j(y)
\longrightarrow
\iint K_R^{s,m}(x-y)\,d\sigma(x)\,d\sigma(y).
\]
By the definition of $C_R^{s,m}(A)$,
\[
\frac1{C_R^{s,m}(A)}
\le
\iint K_R^{s,m}(x-y)\,d\sigma_j(x)\,d\sigma_j(y)
\]
for every $j$. Letting $j\to\infty$ gives
\[
\frac1{C_R^{s,m}(A)}
\le
\iint K_R^{s,m}(x-y)\,d\sigma(x)d\sigma(y).
\]
Multiplying by $\nu(A)^2$ yields \eqref{eq:finite-energy-capacity}.

The monotonicity statement follows because every finitely supported probability on $A$ is also one on $B$.
\end{proof}

The following proposition shows that the dimension profiles enjoy the same monotonicity and continuity properties as the macroscopic intermediate dimensions.
\begin{proposition}\label{prop:profile-shape}
Let $E\subset\R^d$ and $m\in\{1,\ldots,d\}$. Both $\theta\mapsto\lD_\theta^m(E)$ and
$\theta\mapsto\uD_\theta^m(E)$ are nondecreasing. If
$0<\theta<\varphi\le1$, then
\[
 0\le \lD^m_\varphi(E)-\lD^m_\theta(E)
 \le\left(1-\frac{\theta}{\varphi}\right)\bigl(m-\lD^m_\theta(E)\bigr).
\]
and
\[
 0\le \uD^m_\varphi(E)-\uD^m_\theta(E)
\le\left(1-\frac{\theta}{\varphi}\right)\bigl(m-\uD^m_\theta(E)\bigr).
\]
\end{proposition}
\begin{proof}
Increasing $\theta$ reduces $R=2^{n-N\theta}$. By the first comparison
in \eqref{eq:kernel-comparisons}, this decreases the kernel and its
minimum energy, and hence increases capacity. Thus
$\Ct_N^{s,\theta,m}(E)$ is nondecreasing in $\theta$ for every $s$ and
$N$. The definitions of the two dimension profiles then give their monotonicity.

We prove the lower estimate; the upper one is identical with
$\limsup$ in place of $\liminf$. Fix $0<\theta<\varphi\le1$ and $0<s<m$, and put
\[
 t=m-\frac{\theta}{\varphi}(m-s),\quad
 R_\theta=2^{n-N\theta},\quad R_\varphi=2^{n-N\varphi}.
\]
Then $s<t<m$, $1\le R_\varphi\le R_\theta$, and
$\varphi(m-t)=\theta(m-s)$. Writing $r=|z|$, we obtain
\[
 \frac{K_{R_\theta}^{s,m}(z)}{K_{R_\varphi}^{t,m}(z)}
 =\begin{cases}
 1,&r\le1,\\
 r^{t-s},&1<r\le R_\varphi,\\
 r^{m-s}R_\varphi^{t-m},&R_\varphi<r\le R_\theta,\\
 R_\theta^{m-s}R_\varphi^{t-m},&r>R_\theta.
 \end{cases}
\]
The last two expressions are at most
$
 R_\theta^{m-s}R_\varphi^{t-m}
 =2^{n(t-s)-N[\theta(m-s)-\varphi(m-t)]}
 =2^{n(t-s)}.
$
The same bound holds for the first two expressions, since
$R_\varphi\le2^n$. Consequently,
$K_{R_\theta}^{s,m}(z)\le2^{n(t-s)}K_{R_\varphi}^{t,m}(z).$
Integrating against $\mu\times\mu$, taking the infimum over $\mu\in\Pf(E\cap S_n)$, and then taking reciprocals gives $
 2^{-nt}C_{R_\varphi}^{t,m}(E\cap S_n)
 \le2^{-ns}C_{R_\theta}^{s,m}(E\cap S_n).$
For empty shells both sides are zero. Summing over $n\ge N$ yields
\begin{equation}\label{capbds}
\Ct_N^{t,\varphi,m}(E)\le\Ct_N^{s,\theta,m}(E).
\end{equation}

Let $D^m$ denote either dimension profile. If $D^m_\theta(E)<m$, choose
$D^m_\theta(E)<s<m$. The definition yields an exponent \(0<u<s\) at which the corresponding logarithmic growth rate is nonpositive. The comparison
$
\Ct_N^{s,\theta,m}(E)
\le 2^{-N\theta(s-u)}\Ct_N^{u,\theta,m}(E)
$ then shows that the corresponding lower or upper logarithmic growth rate at \(s\) is at most \(-\theta(s-u)<0\). Thus \(s\) also satisfies the defining condition.
The inequality \eqref{capbds} gives the same conclusion at $t$ and parameter
$\varphi$, whence $D^m_\varphi(E)\le t$.
Letting $s\downarrow D^m_\theta(E)$ gives
$
 D^m_\varphi(E)\le
 m-\frac{\theta}{\varphi}\bigl(m-D^m_\theta(E)\bigr),
$
which is the required upper bound. If $D^m_\theta(E)=m$,
monotonicity and the bound $D^m_\varphi(E)\le m$ give equality.
\end{proof}

\subsection{Covering functionals and capacity estimates}\label{sec:mixed}
For $R\ge1$ and $0<v<m$ define a cost on radii $r\ge1$ by
\[
 \psi_R^{v,m}(r)=
 \begin{cases}r^v,&1\le r\le R,\\ r^mR^{v-m},&r>R,\end{cases}
\]

For a nonempty bounded set $A\subset\R^d$, define
\[
 \Mt_R^{v,m}(A)=
 \inf\left\{\sum_{i=1}^k\psi_R^{v,m}(r_i):
 \begin{array}{l}
 k\in\N,\quad A\subset\displaystyle\bigcup_{i=1}^k B(x_i,r_i),\\
 x_i\in A,\quad r_i\ge1
 \end{array}\right\}.
\]
Set $\Mt_R^{v,m}(\varnothing)=0$. Allowing countable covers gives the same value, since
$\psi_R^{v,m}(r)\ge1$ for $r\ge1$.
The identity $\psi_R^{v,m}(r)=1/K_R^{v,m}(r)$ holds for $r\ge1$.
For a fixed target dimension $m$, we also write $\psi_R^v$ and $\Mt_R^v$.
The term \(r^mR^{v-m}\) corresponds to covering an \(m\)-dimensional projected ball of radius comparable to \(r\) by balls of radius \(R\). It should not be interpreted as subdividing a \(d\)-dimensional source ball when \(m<d\).

\begin{lemma}\label{lem:mixed}
If $A\subset \R^d$ is  bounded and nonempty, $R\ge1$ and $0<u<v<m$, then
\begin{equation}
 \Mt_R^{v,m}(A)\le
 C_{u,v,m}R^{v-u}
 \left(1+\log_+\frac{2+\diam A}{R}\right)C_R^{u,m}(A).
 \label{eq:mixed-general}
\end{equation}
If $A\subset S_n$, $n\ge N$ and $R=2^{n-N\theta}$, then
\begin{equation}
 2^{-nv}\Mt_R^{v,m}(A)
 \le C_{u,v,m}(N+1)2^{-N\theta(v-u)}
             2^{-nu}C_R^{u,m}(A).
 \label{eq:mixed-outer}
\end{equation}
\end{lemma}
\begin{proof}
First suppose that $A$ is finite. The probability measures on $A$ form
a compact simplex, and the energy is continuous on it, so there is a
minimizer $\mu$. Write
\[
 I=\iint K_R^{u,m}(x-y)\,d\mu(x)d\mu(y)
   =\frac1{C_R^{u,m}(A)},\quad
 U^\mu(x)=\int K_R^{u,m}(x-y)\,d\mu(y).
\]
For $x\in A$, the measures $(1-\varepsilon)\mu+\varepsilon\delta_x$
are admissible for $0\le\varepsilon\le1$. The right derivative of their
energy at $\varepsilon=0$ is $2(U^\mu(x)-I)$ and must be nonnegative.
Hence $U^\mu(x)\ge I>0$ for every $x\in A$.

Put $D=\diam A$, $L=\lceil\log_2(2+D)\rceil$, and
$\rho_j=2^j$ for $0\le j\le L$. Decompose the integral defining
$U^\mu(x)$ into $|x-y|\le1$ and
$2^{j-1}<|x-y|\le2^j$. The definition of the kernel gives
$K_R^{u,m}(\rho/2)\le2^mK_R^{u,m}(\rho)$, so
\[
 I\le U^\mu(x)\le
 2^m\sum_{j=0}^L K_R^{u,m}(\rho_j)
                  \mu(\overline B(x,\rho_j)).
\]
Let $\delta=v-u>0$. The sum of the corresponding kernel--cost products is
\[
 \begin{split}
 W&:=\sum_{j=0}^L K_R^{u,m}(\rho_j)\psi_R^{v,m}(\rho_j)\\
  &=\sum_{\rho_j\le R}\rho_j^\delta
    +\sum_{\rho_j>R}R^\delta\\
  &\le C_{u,v,m}R^{v-u}
       \left(1+\log_+\frac{2+D}{R}\right).
 \end{split}
\]
Here and in the sequel $\log_+a=\max\{0,\log a\}$: the first sum is geometric, and
the number of terms in the second is at most
$C(1+\log_+((2+D)/R))$, since $2^L\le2(2+D)$.
The preceding potential bound implies
$
 I\le2^mW\max_{0\le j\le L}
       \frac{\mu(\overline B(x,\rho_j))}{\psi_R^{v,m}(\rho_j)}.
$
Thus, for each $x\in A$, we may choose a radius $\rho_x$ such that
\[
 \mu(\overline B(x,\rho_x))\ge
       \frac{I}{2^mW}\psi_R^{v,m}(\rho_x).
\]
Select the auxiliary \emph{closed} balls in nonincreasing order of radius, keeping a ball only when it is disjoint from every previously selected closed ball. Write the selected balls as $\overline B_i=\overline B(x_i,\rho_i)$. If the ball associated with $x\in A$ is discarded, it meets a selected ball with radius $\rho_i\ge\rho_x$. Hence
$
 |x-x_i|\le\rho_x+\rho_i\le2\rho_i<3\rho_i.
$
The selected centres satisfy the same covering conclusion. Thus the \emph{open} balls $B(x_i,3\rho_i)$ cover $A$. Since
$\psi_R^{v,m}(3r)\le3^m\psi_R^{v,m}(r)$, the corresponding mixed covering sum satisfies
\[
 \Mt_R^{v,m}(A)
 \le3^m\sum_i\psi_R^{v,m}(\rho_i)\le\frac{6^mW}{I}\sum_i\mu(\overline B_i)
 \le\frac{6^mW}{I}.
\]
The last inequality uses disjointness of the closed balls. In particular, atoms on their boundaries are not lost. The estimate for $W$ proves \eqref{eq:mixed-general} when $A$ is finite.

For a general bounded nonempty $A$, choose a finite $1/4$-net
$F\subset A$, so every point of $A$ is within $1/4$ of $F$.
Apply the finite-set argument to $F$ and enlarge every covering radius
$r\ge1$ by $1/2$. The enlarged balls cover $A$, still have centres in
$A$, and satisfy
$
 \psi_R^{v,m}(r+1/2)\le(3/2)^m\psi_R^{v,m}(r).
$
Since $\diam F\le\diam A$ and
$C_R^{u,m}(F)\le C_R^{u,m}(A)$, this proves
\eqref{eq:mixed-general} in general, after adjusting the constant.

Finally, if $A\subset S_n$ and $R=2^{n-N\theta}$, then
$\diam A\le2^{n+1}$ and
\[
 1+\log_+\frac{2+\diam A}{R}
 \le1+\log(4\cdot2^{N\theta})\le C(N+1).
\]
Multiplying \eqref{eq:mixed-general} by $2^{-nv}$ and using $2^{-nv}R^{v-u}=2^{-N\theta(v-u)}2^{-nu}$
gives \eqref{eq:mixed-outer}.
\end{proof}

For a bounded set $A\subset\R^d$, let $\Ht_R^s(A)$ be the infimum of
$\sum_i r_i^s$ over covers with centres in $A$ and $1\le r_i\le R$.
If $A\subset S_n$, allowing centres anywhere in $S_n$ changes this
infimum by at most a dimensional constant. Indeed, for each ball
$B(x,r)$ of such a cover, a maximal $(r/2)$-separated subset of
$A\cap B(x,r)$ has at most $C_d$ points; the radius-$r$ balls centred
at these points cover $A\cap B(x,r)$ and still have legal radii.

\begin{lemma}\label{lem:ambient-cover}
If $0<s<m\le d$, then
\begin{equation}
 2^{-m} C_R^{s,m}(A)\le\Ht_R^s(A).
 \label{eq:cap-below-cover}
\end{equation}
If $0<u<v<d$, then
\begin{equation}
  \mathcal{H}_{n,N}^{v,\theta}(E)
 \le C_{d,u,v}(N+1)2^{-N\theta(v-u)}
      2^{-nu}C_{2^{n-N\theta}}^{u,d}(E\cap S_n).
 \label{eq:ambient-cover-above}
\end{equation}
\end{lemma}
\begin{proof}
For the first assertion, fix an admissible cover
$A\subset\bigcup_i B(x_i,r_i)$, where $1\le r_i\le R$, and
$\mu\in\Pf(A)$. Partition the finite support of $\mu$ into
disjoint sets $A_i\subset B(x_i,r_i)$, assigning each point to the
first ball containing it. For $x,y\in A_i$, we have $|x-y|<2r_i$.
Since the kernel is nonincreasing in distance,
\[
 K_R^{s,m}(x-y)\ge K_R^{s,m}(2r_i)
 =\begin{cases}
 (2r_i)^{-s},&2r_i\le R,\\
 R^{m-s}(2r_i)^{-m},&2r_i>R,
 \end{cases}
 \ge2^{-m}r_i^{-s}.
\]
In the second case we used $R\ge r_i$ and $m>s$.
Discarding the nonnegative cross terms and applying Cauchy--Schwarz gives
\[
 \begin{split}
 \iint K_R^{s,m}\,d\mu\,d\mu
 \ge2^{-m}\sum_i r_i^{-s}\mu(A_i)^2
 \ge2^{-m}\frac{(\sum_i\mu(A_i))^2}{\sum_i r_i^s}
 =\frac{2^{-m}}{\sum_i r_i^s}.
 \end{split}
\]
Taking the infimum over $\mu$ and then reciprocals yields
$C_R^{s,m}(A)\le2^m\sum_i r_i^s$.
The infimum over admissible covers proves
\eqref{eq:cap-below-cover}.

For the second assertion, put $A_n=E\cap S_n$ and
$R=2^{n-N\theta}$. Consider any mixed cover
$A_n\subset\bigcup_i B(x_i,r_i)$ with $x_i\in A_n$ and $r_i\ge1$.
Keep every ball with $r_i\le R$. For a ball with $r_i>R$, choose a
maximal $R/2$-separated subset $Y_i$ of $A_n\cap B(x_i,r_i)$.
Maximality implies
$
 A_n\cap B(x_i,r_i)\subset\bigcup_{y\in Y_i}B(y,R).
$
The balls $B(y,R/4)$, $y\in Y_i$, are pairwise disjoint and lie in
$B(x_i,r_i+R/4)$. Comparing their Euclidean volumes gives
\[
 \#Y_i\le\left(\frac{r_i+R/4}{R/4}\right)^d
 \le5^d\left(\frac{r_i}{R}\right)^d.
\]
Thus the replacement balls contribute at most
$
\#Y_i(R/2^n)^v
\le
5^d2^{-nv}\psi_R^{v,d}(r_i)
$
to the normalized covering sum defining $\mathcal H_{n,N}^{v,\theta}(E)$.
All retained and replacement balls have centres in $A_n\subset S_n$
and radii in $[1,R]$, so they are admissible for the shell content.
After summing their costs, normalizing by $2^{nv}$, and taking the
infimum over mixed covers, we obtain
$
 \mathcal H_{n,N}^{v,\theta}(E)
 \le C_d2^{-nv}\Mt_R^{v,d}(A_n).
$
Apply \eqref{eq:mixed-outer} with $m=d$ to conclude
\eqref{eq:ambient-cover-above}. Empty sets or shells give zero on both
sides and require no separate argument.
\end{proof}

\subsection{Endpoint identities and profile bounds}
We use the energy theorem of Daw and Seuret \cite[Theorem~2.7(2)]{DS}.
For a Borel set $E\subset\R^d$ and $0<s<\hDim E$, there is a Radon
measure $\mu$, concentrated on $E$, with
\begin{equation}
 \mu(E)=\infty,\quad
 \sum_{n\ge0}2^{ns}I_s(\mu|_{S_n})<\infty,
 \quad I_s(\nu)=\iint(1\vee|x-y|)^{-s}\,d\nu(x)d\nu(y).
 \label{eq:DS-energy}
\end{equation}
Indeed, choose $s<q<\hDim E$. The macroscopic $q$-content series
diverges, so the cited theorem, with loss parameter $\varepsilon=q-s$,
gives a Radon measure $\nu$ with $\nu(E)=\infty$ and finite weighted
$s$-energy as in \eqref{eq:DS-energy}. Restricting to $\mu=\nu|_E$
preserves infinite mass and can only decrease each energy; the restricted
measure is Radon since $E$ is Borel. The cited theorem is valid in $\R^d$.
The corresponding high-dimensional projection formula will be proved in
Section~\ref{sec:projection}, rather than assumed here.

\begin{proposition}\label{prop:recovery}
For Borel $E\subset\R^d$, $m\in\{1,\ldots,d\}$, and $0\le\theta\le1$,
\begin{equation}
 \lD^m_0(E)=\uD^m_0(E)=\min\{\hDim E,m\},\label{eq:cap-zero}
 \end{equation}
 
 \begin{equation}
 \lD^d_\theta(E)=\Dt^-_\theta(E),\quad
 \uD^d_\theta(E)=\Dt^+_\theta(E),\label{eq:ambient-recovery}
  \end{equation}
  
  \begin{equation}
  \max\left\{\min(\hDim E,m),\frac {m}{d}\Dt^-_\theta(E)\right\}\le\lD^m_\theta(E)\le\min\{\Dt^-_\theta(E),m\}
 \label{eq:source-upper}
 \end{equation}
and
 \[
 \max\left\{\min(\hDim E,m),\frac {m}{d}\Dt^+_\theta(E)\right\}\le
 \uD^m_\theta(E)\le\min\{\Dt^+_\theta(E),m\}.
\]
\end{proposition}
\begin{proof}
We first prove \eqref{eq:cap-zero}.
Fix $0<s<m$ and set
$
 b_n=2^{-ns}C_{2^n}^{s,m}(E\cap S_n).
$
Then
$
 \Ct_N^{s,0,m}(E)=\sum_{n\ge N}b_n.
$
If $\hDim E<m$, choose $s$ with $\hDim E<s<m$. Lemma~\ref{lem:ambient-cover}
and the definition of the macroscopic Hausdorff dimension give
$
 \sum_{n\ge1}b_n\lesssim \sum_{n\ge1}\mathcal H_n^s(E)<\infty.
$
Hence the tails $\Ct_N^{s,0,m}(E)$ tend to zero, so
$\uD_0^m(E)\le s$. Letting $s\downarrow\hDim E$ gives the
required upper bound. If $\hDim E\ge m$, the same upper bound follows
immediately from $\lD_0^m(E),\uD_0^m(E)\le m$.

For the reverse inequality, let
$0<s<\min\{\hDim E,m\}$. By \eqref{eq:DS-energy}, there exists a Radon
measure $\mu$ concentrated on $E$ such that
$
 \mu(E)=\infty
 $ and 
 $
 \sum_{n\ge1}2^{ns}I_s(\mu|_{S_n})<\infty.
$
Write $a_n=\mu(S_n)$ for $n\ge1$. Since $\mu(S_0)<\infty$, one still has $\sum_{n\ge1}a_n=\infty$. Since $x,y\in S_n$ implies
$|x-y|\le 2^{n+1}$, the kernel $K_{2^n}^{s,m}(x-y)$ is comparable to
$(1\vee |x-y|)^{-s}$, with constants depending only on $s$ and $m$.
Lemma~\ref{lem:finite-energy} therefore gives, whenever $a_n>0$,
$
 2^{ns}I_s(\mu|_{S_n})\gtrsim \frac{a_n^2}{b_n}.
$
If $\sum_n b_n<\infty$, then Cauchy--Schwarz yields
\[
 \sum_{n\ge1}a_n
 \lesssim
 \left(\sum_{n\ge1}b_n\right)^{1/2}
 \left(\sum_{n\ge1}2^{ns}I_s(\mu|_{S_n})\right)^{1/2}
 <\infty,
\]
contradicting $\mu(E)=\infty$. Thus $\sum_n b_n=\infty$, so
$\Ct_N^{s,0,m}(E)=\infty$ for every $N$. Consequently
$\lD_0^m(E)\ge s$. Letting
$s\uparrow\min\{\hDim E,m\}$ proves \eqref{eq:cap-zero}.

We next prove \eqref{eq:ambient-recovery}.
By Lemma~\ref{lem:ambient-cover},
$
 \Ct_N^{s,\theta,d}(E)\lesssim \Ht_N^{s,\theta}(E),
$
and hence
$
 \lD_\theta^d(E)\le \Dt^-_\theta(E).
$
Assume now that $\theta>0$. If $\lD_\theta^d(E)=d$, the reverse
inequality is automatic. Otherwise choose
$
 \lD_\theta^d(E)<u<v<d.
$
Summing \eqref{eq:ambient-cover-above} over $n\ge N$ gives
$
 \Ht_N^{v,\theta}(E)
 \lesssim
 (N+1)2^{-N\theta(v-u)}\Ct_N^{u,\theta,d}(E).
$
Since $u>\lD_\theta^d(E)$, the lower logarithmic growth rate of
$\Ct_N^{u,\theta,d}(E)$ is nonpositive. Therefore
\[
 \liminf_{N\to\infty}
 \frac{\log \Ht_N^{v,\theta}(E)}{N\log2}
 \le -\theta(v-u)<0,
\]
so $\Dt^-_\theta(E)\le v$. Letting
$v\downarrow\lD_\theta^d(E)$ gives
$
 \lD_\theta^d(E)=\Dt^-_\theta(E).
$
The proof of
$\uD_\theta^d(E)=\Dt^+_\theta(E)$ is identical, with $\limsup$ in
place of $\liminf$ (and is again immediate if $\uD_\theta^d(E)=d$\BlueChange{)}.
When $\theta=0$, both identities follow from
\eqref{eq:cap-zero} and Proposition~\ref{prop:basic-endpoints}. This proves
\eqref{eq:ambient-recovery}.

Finally, \BlueChange{we} prove \eqref{eq:source-upper}.
First, Lemma~\ref{lem:ambient-cover} gives, for $0<s<m$,
$
 \Ct_N^{s,\theta,m}(E)\lesssim \Ht_N^{s,\theta}(E).
$
It follows that
$
 \lD_\theta^m(E)\le \Dt^-_\theta(E).
$
Since $\lD_\theta^m(E)\le m$ by definition,
$
 \lD_\theta^m(E)\le \min\{\Dt^-_\theta(E),m\}.
$

For the first lower bound, Proposition~\ref{prop:profile-shape} shows
that $\theta\mapsto\lD_\theta^m(E)$ is nondecreasing. Together with
\eqref{eq:cap-zero}, this gives
\[
 \lD_\theta^m(E)
 \ge \lD_0^m(E)
 =\min\{\hDim E,m\}.
\]

It remains to prove the lower bound involving the ambient dimension.
Set $\alpha=\frac{m}{d}.$
For $0<t<d$, the definition of the kernels gives the pointwise identity
$
 K_R^{\alpha t,m}(z)=\bigl(K_R^{t,d}(z)\bigr)^\alpha.
$
Because $0<\alpha\le1$, Jensen's inequality implies that, for every
finitely supported probability measure $\mu$ on a bounded set $A$,
\[
 \iint K_R^{\alpha t,m}(x-y)\,d\mu(x)d\mu(y)
 \le
 \left(
 \iint K_R^{t,d}(x-y)\,d\mu(x)d\mu(y)
 \right)^\alpha.
\]
Taking the infimum over $\mu\in\Pf(A)$ and then reciprocals yields $
 C_R^{\alpha t,m}(A)
 \ge \bigl(C_R^{t,d}(A)\bigr)^\alpha.
$
Apply this with $A=E\cap S_n$ and $R=2^{n-N\theta}$. Then
$
 2^{-n\alpha t}C_R^{\alpha t,m}(E\cap S_n)
 \ge
 \left(2^{-nt}C_R^{t,d}(E\cap S_n)\right)^\alpha.
$
Since $0<\alpha\le1$ and
$(\sum_n a_n)^\alpha\le\sum_n a_n^\alpha$ for $a_n\ge0$, summing over
$n\ge N$ gives
\begin{equation}\label{eq:capacity-power-comparison}
 \Ct_N^{\alpha t,\theta,m}(E)
 \ge \bigl(\Ct_N^{t,\theta,d}(E)\bigr)^\alpha.
\end{equation}

If $\Dt^-_\theta(E)=0$, the desired bound is trivial. Otherwise let
$0<t<\Dt^-_\theta(E)$. By the ambient identity already proved,
$t<\lD_\theta^d(E)$, and hence
\[
 \liminf_{N\to\infty}
 \frac{\log\Ct_N^{t,\theta,d}(E)}{N\log2}>0.
\]
It follows from \eqref{eq:capacity-power-comparison} that
\[
 \liminf_{N\to\infty}
 \frac{\log\Ct_N^{\alpha t,\theta,m}(E)}{N\log2}>0,
\]
so
$
 \lD_\theta^m(E)\ge \alpha t=\frac{m}{d}t.
$
Letting $t\uparrow\Dt^-_\theta(E)$ yields
$
 \lD_\theta^m(E)\ge \frac{m}{d}\Dt^-_\theta(E).
$
Combining the two lower bounds with the upper bound proves
\eqref{eq:source-upper}.

The final displayed estimate for the upper profile is proved in the
same way, replacing $\liminf$ by $\limsup$ throughout and using
$\uD_\theta^d(E)=\Dt^+_\theta(E)$ in place of the lower-profile
ambient identity.
\end{proof}

\section{Proof of Theorem \ref{thm:projection}}\label{sec:projection}
Throughout this section $E\subset\R^d$ is Borel, $m\in \{1,\ldots,d\}$, and
$A_n=E\cap S_n$. We abbreviate $K_R^{s,m}$ and $C_R^{s,m}$ to $K_R^s$
and $C_R^s$. All integrals over $G(d,m)$ use $\gamma_{d,m}$.
Covering definitions apply to arbitrary target sets; no Borel or local
finiteness assumption on $P_VE$ is used.

\subsection{Grassmannian estimates}
We first record the two estimates that replace angular integration.
For $\ell\ge1$ and $0<s<m$, put
$Q_\ell^s(v)=(1\vee|v|)^{-s}\ind_{\{|v|\le \ell\}}$ on $V\in G(d,m)$.

\begin{lemma}\label{lem:grassmann}
For $z\ne0$ and $a>0$,
\begin{equation}
 \gamma_{d,m}\{V:|P_Vz|\le a\}
 \le C_{d,m}\min\{1,(a/|z|)^m\}.
 \label{eq:grassmann-small-ball}
\end{equation}
For $\ell\ge1$, $0<s<m$, and $z\in\R^d$,
\begin{equation}
 \int_{G(d,m)}Q_\ell^s(P_Vz)\,d\gamma_{d,m}(V)
 \le C_{d,m,s}K_\ell^{s,m}(z).
 \label{eq:truncated-average}
\end{equation}
\end{lemma}
\begin{proof}
The case $m=d$ is immediate, since $G(d,d)$ consists of the single
subspace $\R^d$. Assume \BlueChange{$m<d$}. By rotational invariance, we may take $z=|z|e_1$. Let
$
V_0=\operatorname{span}\{e_1,\ldots,e_m\}.
$
A random $V\in G(d,m)$ may be represented as $V=UV_0$, where $U$ is distributed according to Haar measure on the orthogonal group. Hence
$
\frac{|P_Vz|}{|z|}
=|P_Ve_1|
=|P_{V_0}(U^{-1}e_1)|.
$
Since $U^{-1}e_1$ is uniformly distributed on $S^{d-1}$, if
$
X=(X_1,\ldots,X_d)
$
is uniform on $S^{d-1}$, then
$\frac{|P_Vz|}{|z|}$
has the same distribution as $
(X_1^2+\cdots+X_m^2)^{1/2}.$
Denote the distribution of $(X_1,\ldots,X_m)$ on the unit ball of $\mathbb R^m$ by $\sigma$.   Its density is
bounded on the ball $B(0,1/2)$.  Hence, for $0<a\le |z|/2$,
$
 \gamma_{d,m}\{V:|P_Vz|\le a\}
 \le C_{d,m}(a/|z|)^m.
$
If $a>|z|/2$, enlarging the constant gives
\eqref{eq:grassmann-small-ball} as well.

We next prove \eqref{eq:truncated-average}.  Put $R_0=|z|$.  When
$1<R_0\le2\ell$, split the above coordinate distribution into the
regions $|y|\le1/2$ and $|y|>1/2$.  On the first region the density is
bounded, and the change of variables $u=R_0|y|$ gives
\[
 \int_{|y|\le1/2}(1\vee R_0|y|)^{-s}\,d\sigma(y)
 \le C R_0^{-m}\int_0^{R_0/2}(1\vee u)^{-s}u^{m-1}\,du
 \le C R_0^{-s},
\]
because $s<m$.  On $|y|>1/2$ the integrand is at most
$C R_0^{-s}$.  Thus
\[
 \int_{G(d,m)}(1\vee|P_Vz|)^{-s}\,d\gamma_{d,m}(V)
 \le C(1\vee|z|)^{-s}.
\]
For $|z|\le2\ell$ this is bounded by a constant multiple of
$K_\ell^{s,m}(z)$ (also when $\ell<|z|\le2\ell$).

Finally, if $|z|>2\ell$, the indicator in $Q_\ell^s$ restricts the
coordinate variable to $|y|\le\ell/|z|<1/2$.  Using the bounded density
again,
\[
 \begin{aligned}
 \int Q_\ell^s(P_Vz)\,d\gamma_{d,m}(V)
 &\le C|z|^{-m}\int_0^\ell(1\vee r)^{-s}r^{m-1}\,dr\\
 &\le C_{d,m,s}\ell^{m-s}|z|^{-m}
 =C_{d,m,s}K_\ell^{s,m}(z).
 \end{aligned}
\]
The case $|z|\le1$ is immediate.
\end{proof}

\subsection{A fractional-moment estimate for the whole tail}
The next lemma treats every source shell and includes $\theta=0$.
It separates projected balls close to the origin from the remaining balls.

\begin{lemma}\label{lem:moment}
Fix $N\ge1$, $0\le\theta\le1$ and $0<t<s<m$, and put $p=t/s$.
For each $n\ge N$, choose finitely many balls $B(x_{n,i},r_{n,i})$
covering $A_n$, with $x_{n,i}\in A_n$ and $r_{n,i}\ge1$. Suppose
\[
 R_n=2^{n-N\theta},\quad
 G_N=\sum_{n\ge N,i}2^{-nt}\psi_{R_n}^{t,m}(r_{n,i})<\infty.
\]
There are a measurable set $\mathcal B_N\subset G(d,m)$ and a
nonnegative measurable function $W_N$ such that
\begin{align}
 \gamma_{d,m}(\mathcal B_N)&\le C2^{-N\theta(m-t)}G_N,
       \label{eq:bad-directions}\\
 \mathcal{H}_N^{s,\theta}(P_VE)&\le W_N(V)
       \quad(V\notin\mathcal B_N),\label{eq:moment-majorant}\\
 \int_{G(d,m)}W_N(V)^p\,d\gamma_{d,m}(V)&\le C_{d,m,s,t}G_N.
       \label{eq:fractional-moment}
\end{align}
The constants do not depend on $N$, $\theta$, or the largest source index.
\end{lemma}
\begin{proof}
Let $\mathcal B_N$ be the union, over all selected source balls, of the
events
$
 \{V:|P_Vx_{n,i}|\le2r_{n,i}\}.
$
Fix one such ball and write $x=x_{n,i}$, $r=r_{n,i}$ and
$\xi=r/2^n$.  Since $x\in S_n$, \eqref{eq:grassmann-small-ball} gives
$
 \gamma_{d,m}\{V:|P_Vx|\le2r\}
 \le C_{d,m}\min\{1,\xi^m\}.
$
Moreover,
\begin{equation}
 \xi^m\le
 2^{-N\theta(m-t)}2^{-nt}\psi_{R_n}^{t,m}(r).
 \label{eq:moment-bad-ball}
\end{equation}
Indeed, if $r\le R_n$, then
$\xi\le2^{-N\theta}$ and the right-hand side equals
$2^{-N\theta(m-t)}\xi^t$; if $r>R_n$, the two sides are equal.
Summing \eqref{eq:moment-bad-ball} over the source balls proves
\eqref{eq:bad-directions}.

We now estimate, for each target shell \(S_j^V\), the contribution of the projected ball \(B(P_Vx,r)\) to an admissible covering of \(P_VE\cap S_j^V\).
Fix a target shell $N\le j\le n$ and put
$u_j=2^{j-N\theta}$.  If $B(P_Vx,r)$ meets $S_j^V$ and $r\le2^j$, then
$|P_Vx|\le2^j+r\le2^{j+1}$; hence
\begin{equation}
 \gamma_{d,m}\{V:B(P_Vx,r)\cap S_j^V\ne\varnothing\}
 \le C_{d,m}2^{-m(n-j)}.
 \label{eq:moment-hit-shell}
\end{equation}
On the retained event $|P_Vx|>2r$, no intersection is possible when
$r>2^j$.

For $j\ge1$, let
\[
 I_j(V)=\ind_{\{|P_Vx|>2r\}}
        \ind_{\{|P_Vx|<2^j+r\}}
        \ind_{\{|P_Vx|+r>2^{j-1}\}}.
\]
This is the indicator that the projected open ball meets $S_j^V$ and the centre is retained. It is Borel because $V\mapsto|P_Vx|$ is continuous. We set $w_j=0$ when $r>2^j$; in that case $I_j=0$.  If $r\le u_j$, a maximal $r$-separated subset
of $B(P_Vx,r)\cap S_j^V$ has bounded cardinality, and the corresponding
radius-$r$ balls, with centres in $S_j^V$, cover the intersection.  We
may therefore take
$
 w_j(V)=C_{m,s}(r/2^j)^s I_j(V).
$
Using \eqref{eq:moment-hit-shell} and $p=t/s$,
\[
 \int w_j^p\,d\gamma_{d,m}
 \le C\xi^t2^{-(m-t)(n-j)}.
\]
The sum over $j$ is $O(\xi^t)$.  Notice that this case can occur only
when $r\le R_n$, because $u_j\le u_n=R_n$.

Suppose next that $u_j<r\le2^j$.  A maximal $u_j$-separated subset of
the same intersection gives at most $C_m(r/u_j)^m$ radius-$u_j$ balls.
Thus these balls contribute at most
$
C_{m,s}\left(r/2^j\right)^m 2^{N\theta(m-s)}
$
to the covering sum defining $\mathcal H_{j,N}^{s,\theta}(P_VE)$
and consequently
\begin{equation}
 \int w_j^p\,d\gamma_{d,m}
 \le C\xi^{mp}2^{N\theta(mp-t)}
          2^{-m(1-p)(n-j)}.
 \label{eq:moment-second-case}
\end{equation}
Here $u_j<r$ implies $n-j>h$, where
$h=\log_2(R_n/r)$. If $r\le R_n$, then $h\ge0$, and the geometric sum is bounded by
$
 \sum_{n-j>h}2^{-m(1-p)(n-j)}
 \le C_{m,p}2^{-m(1-p)h}.
$
Substitution into \eqref{eq:moment-second-case} gives
$
 C\xi^{mp}2^{N\theta(mp-t)}2^{-m(1-p)h}
 =C\xi^m2^{N\theta(m-t)}\le C\xi^t,
$
since $\xi\le2^{-N\theta}$.  If $r>R_n$, we may sum over all
$N\le j\le n$ and obtain
$
 C\xi^{mp}2^{N\theta(mp-t)}
 \le C\xi^m2^{N\theta(m-t)},
$
since $\xi2^{N\theta}>1$ and $p-1<0$.  Combining the two cases yields
\begin{equation}
 \sum_{j=N}^n\int w_j^p\,d\gamma_{d,m}
 \le C_{d,m,s,t}2^{-nt}\psi_{R_n}^{t,m}(r).
 \label{eq:moment-one-ball}
\end{equation}

Define $W_N$ by summing these $w_j$ over all source balls and all
relevant target shells.  Since $0<p<1$,
$(\sum a_i)^p\le\sum a_i^p$ for nonnegative $a_i$.  Tonelli's theorem
and \eqref{eq:moment-one-ball} therefore give
\eqref{eq:fractional-moment}.

It remains to check that $W_N$ really majorizes the target content.
Take $V\notin\mathcal B_N$ and a point $y\in E$ whose projection lies
in $S_j^V$ with $j\ge N$.  If $y\in S_n$, then
$|P_Vy|\le|y|<2^n$, so necessarily $n\ge j$.  Choose a source ball
covering $y$.  Its centre satisfies the retained condition, its
projected ball meets $S_j^V$, and the construction above supplies a
legal target cover with radius at most $u_j$.  Hence all target shells
$j\ge N$ are covered and
$\mathcal H_N^{s,\theta}(P_VE)\le W_N(V)$.  If $W_N(V)<\infty$, each
individual shell cover is finite because every legal ball has radius at
least one and hence normalized cost at least $2^{-js}$.  This proves
\eqref{eq:moment-majorant} and completes the proof.
\end{proof}

The removal of $\mathcal B_N$ avoids an uncontrolled shell-count factor, since a large projected ball containing the origin may intersect many target shells and contribute significantly to each of them.

\subsection{The upper estimate for $0<\theta\le 1$}
Fix $\theta>0$ and one sign.  If
$\Dpm^{\pm,m}_\theta(E)=m$, the ambient bound
$\Dt^\pm_\theta(P_VE)\le m$ gives the result.  Otherwise choose
\[
 \Dpm^{\pm,m}_\theta(E)<u<t<s<m,
 \quad 0<\varepsilon<\theta(t-u).
\]
By the exponent comparison preceding Proposition~\ref{prop:profile-shape},
all exponents strictly above the profile are admissible.  Hence, for the
lower sign, there is a strictly increasing sequence $\mathcal S$ such
that
$\Ct_N^{u,\theta,m}(E)\le2^{\varepsilon N}$ for 
$N\in\mathcal S,$
whereas for the upper sign the same estimate holds for every sufficiently
large $N$.

For each selected $N$, Lemma~\ref{lem:mixed} supplies, shell by shell,
finite mixed covers chosen independently of $V$.  Empty shells are omitted. For every nonempty shell the mixed content is positive, and we choose a cover with cost at most twice its infimum.  Summing
\eqref{eq:mixed-outer} over $n\ge N$ gives
\[
 G_N
 \le C_{u,t,m}(N+1)2^{-N\theta(t-u)}
       \Ct_N^{u,\theta,m}(E)\le C_{u,t,m}(N+1)
       2^{-N[\theta(t-u)-\varepsilon]}.
\]
Since $\theta(t-u)-\varepsilon>0$, after decreasing the exponent slightly
we obtain $G_N\le2^{-\delta N}$ for some $\delta>0$ and all sufficiently
large selected scales.

Apply Lemma~\ref{lem:moment}.  Because
$2^{-N\theta(m-t)}\le1$, Markov's inequality and
\eqref{eq:fractional-moment} imply
\[
 \gamma_{d,m}(\mathcal B_N)
 +\gamma_{d,m}\{V:W_N(V)>1\}
 \le C2^{-\delta N}.
\]
The series of these exceptional probabilities is finite, also when the
sum is restricted to $N\in\mathcal S$.  By Borel--Cantelli, for almost
every $V$ and all sufficiently large selected $N$,
$
 \mathcal H_N^{s,\theta}(P_VE)\le W_N(V)\le1.
$
For the lower sign this gives a subsequence along which the normalized
logarithmic growth at exponent $s$ is nonpositive; for the upper sign it
gives the same conclusion at every sufficiently large scale.  Thus
$\Dt^\pm_\theta(P_VE)\le s$.  Finally choose a countable sequence of
$s$ decreasing to $\Dpm^{\pm,m}_\theta(E)$ and intersect the resulting
full-measure sets.  We obtain
\begin{equation}
 \Dt^\pm_\theta(P_VE)\le\Dpm^{\pm,m}_\theta(E)
 \quad\text{for almost every }V\in G(d,m).
 \label{eq:projection-upper}
\end{equation}

\subsection{The lower estimate by capacity weights}
Fix $0\le\theta\le1$ and
$0<s<\Dpm^{\pm,m}_\theta(E)$.  For the lower profile, the defining
liminf at exponent $s$ is strictly positive; for the upper profile, the
defining limsup is strictly positive.  Hence there is $b>0$ such that
$
 \Ct_N^{s,\theta,m}(E)\ge2^{2bN}
$
for every sufficiently large $N$ in the lower-sign case, and along an
increasing sequence of $N$ in the upper-sign case.  When $\theta=0$,
Proposition~\ref{prop:recovery} shows that the relevant capacity series
diverges, so every tail is infinite and the same conclusion is available.

At each selected scale choose finitely many nonempty source shells
$I_N\subset\{n:n\ge N\}$ such that
\[
 B_N:=\sum_{n\in I_N}b_{N,n}\ge2^{bN},\quad
 b_{N,n}=2^{-ns}C_{R_{N,n}}^s(A_n),\quad
 R_{N,n}=2^{n-N\theta}.
\]
This is possible because a finite partial sum of a nonnegative series can
be made larger than $2^{bN}$.  No bound on $\#I_N$ or on its largest
index is used.  For each $n\in I_N$, choose
$\mu_{N,n}\in\Pf(A_n)$ whose $K_{R_{N,n}}^s$-energy is at most
twice the minimum.  Put
$L_N=\lceil4\log_2(N+1)\rceil$ and $ M_N=N-L_N,
$
and restrict to large $N$ for which $M_N\ge1$.

We first discard the directions for which too much of a source shell is
projected close to the origin.  For $x\in S_n$,
\eqref{eq:grassmann-small-ball} gives
$
 \gamma_{d,m}\{V:|P_Vx|<2^{n-L_N}\}
 \le C_{d,m}2^{-mL_N}.
$
Let
$
 \ell_{N,n}(V)=
 \mu_{N,n}\{x:|P_Vx|<2^{n-L_N}\}.
$
Tonelli's theorem yields
$\int\ell_{N,n}(V)\,d\gamma_{d,m}(V)\le C2^{-mL_N}$.
By Markov's inequality,
$
 \gamma_{d,m}\{V:\ell_{N,n}(V)>1/2\}
 \le C_{d,m}(N+1)^{-4}.
$
Thus, outside this exceptional set, at least half of the projected mass
lies in the radial region $|y|\ge2^{n-L_N}$.

Next put $ U_{N,n}=2^{n+1-M_N\theta}$
and let
\[
 J_{N,n}(V)=
 \iint Q_{U_{N,n}}^s(P_V(x-y))\,
       d\mu_{N,n}(x)d\mu_{N,n}(y).
\]
Since
$U_{N,n}/R_{N,n}=2^{1+\theta L_N}$, Lemma~\ref{lem:grassmann} and the
first kernel comparison in \eqref{eq:kernel-comparisons} give
\[
 \begin{aligned}
 \int J_{N,n}(V)\,d\gamma_{d,m}(V)
 &\le C_{d,m,s}
 \iint K_{U_{N,n}}^s(x-y)\,d\mu_{N,n}(x)d\mu_{N,n}(y)\\
 &\le C_{d,m,s}\left(\frac{U_{N,n}}{R_{N,n}}\right)^{m-s}
       \frac{1}{C_{R_{N,n}}^s(A_n)}\\
 &\le \frac{C_{d,m,s}(N+1)^{4m}}
 {C_{R_{N,n}}^s(A_n)}.
 \end{aligned}
\]
A second use of Markov's inequality shows that
\begin{equation}
 J_{N,n}(V)\le
 \frac{C_{d,m,s}(N+1)^{4m+4}}{C_{R_{N,n}}^s(A_n)}
 \label{eq:good-energy}
\end{equation}
fails on a set of measure $O_{d,m,s}((N+1)^{-4})$.

Call $n\in I_N$ \emph{good for $V$} if both the mass-retention estimate
and \eqref{eq:good-energy} hold.  Instead of taking a union over all
$n\in I_N$, which would introduce the uncontrolled factor $\#I_N$, we
weight bad shells by their capacity contributions:
\[
 Z_N(V)=\frac1{B_N}
 \sum_{\substack{n\in I_N\\n\text{ bad for }V}}b_{N,n}.
\]
The two preceding exceptional-set bounds are uniform in $n$, so
\[
 \int Z_N\,d\gamma_{d,m}\le C_{d,m,s}(N+1)^{-4}.
\]
Another Markov estimate gives
\[
 \gamma_{d,m}\{V:Z_N(V)>1/2\}
 \le2C_{d,m,s}(N+1)^{-4}.
\]
The sum over the selected outer scales is finite.  Borel--Cantelli
therefore implies that, for almost every $V$,
\begin{equation}
 \sum_{\substack{n\in I_N\\n\text{ good for }V}}b_{N,n}
 \ge \frac{B_N}{2}
 \label{eq:good-capacity-weight}
\end{equation}
for all sufficiently large selected $N$.

Fix such a direction $V$ and a good shell $n$.  Restrict the projected
measure to the region $|y|\ge2^{n-L_N}$ and denote the restriction by
$\nu$.  Then $\nu(V)\ge1/2$ and, because projection does not increase
norm, its support is contained in the target shells
$n-L_N+1,\ldots,n$.  Take arbitrary legal covers of these shells at
outer scale $M_N$.  Partition the support of $\nu$ into measurable pieces
$A_i$, each assigned to one covering ball $B(y_i,r_i)$.  Since
$j\le n$ for every relevant shell,
$
 \operatorname{diam}A_i\le2r_i
 \le2^{n+1-M_N\theta}=U_{N,n}.
$
For $x,y\in A_i$ we therefore have
$Q_{U_{N,n}}^s(x-y)\ge(2r_i)^{-s}$.  Positivity of the kernel gives
$
 J_{N,n}(V)\ge2^{-s}\sum_i r_i^{-s}\nu(A_i)^2.
$
Since $\sum_i\nu(A_i)=\nu(V)\ge1/2$, Cauchy--Schwarz yields
\[
 \frac14\le
 \left(\sum_i r_i^s\right)
 \left(\sum_i r_i^{-s}\nu(A_i)^2\right).
\]
Combining the last two inequalities, and using $j\le n$ when the costs
are normalized, gives
\[
 \sum_{j=n-L_N+1}^n
 \mathcal H_{j,M_N}^{s,\theta}(P_VE)
 \ge\frac{c_s2^{-ns}}{J_{N,n}(V)}
 \ge c_{d,m,s}(N+1)^{-4m-4}b_{N,n}.
\]

Now sum over the good source shells.  A fixed target shell $j$ can occur
in at most $L_N$ of the intervals
$[n-L_N+1,n]$.  Using \eqref{eq:good-capacity-weight},
\[
 \mathcal H_{M_N}^{s,\theta}(P_VE)
 \ge \frac{c_{d,m,s}}{L_N(N+1)^{4m+4}}
       \sum_{\substack{n\in I_N\\n\text{ good}}}b_{N,n}\ge \frac{c_{d,m,s}2^{bN}}
 {L_N(N+1)^{4m+4}}.
\]
For all sufficiently large $N$, the increment of
$L_N=\lceil4\log_2(N+1)\rceil$ is either zero or one.  Hence
$M_N=N-L_N$ is nondecreasing with increments zero or one, tends to
infinity, and visits every sufficiently large integer.  This proves a
positive liminf growth rate in the lower-sign case.  In the upper-sign
case, $M_N\to\infty$ along the selected sequence; after passing to a
subsequence we may assume that the $M_N$ are strictly increasing, which
gives a positive limsup growth rate.  Since $M_N/N\to1$, in both cases the corresponding target dimension is at least $s$. A strict inequality at this one test exponent is neither asserted nor needed.

Letting $s$ increase to the profile through a countable sequence gives
\begin{equation}
 \Dt^\pm_\theta(P_VE)\ge\Dpm^{\pm,m}_\theta(E)
 \quad\text{for almost every }V\in G(d,m).
 \label{eq:projection-lower}
\end{equation}
This argument includes $\theta=0$.

\subsection{The upper estimate at $\theta=0$}
It remains to prove the upper estimate at $\theta=0$.  If
$\hDim E\ge m$, the ambient bound $\hDim(P_VE)\le m$ is enough.
Assume $\hDim E<m$ and choose
$
 \hDim E<u<t<s<m.
$
By Proposition~\ref{prop:recovery}, the series defining
$\Ct_1^{u,0,m}(E)$ converges, so its tails
$\Ct_N^{u,0,m}(E)$ tend to zero.  Here $R=2^n$, and for
$A_n\subset S_n$ the ratio $(2+\operatorname{diam}A_n)/R$ is uniformly
bounded.  Hence the logarithmic factor in Lemma~\ref{lem:mixed} is
bounded independently of $n$ and $N$.  Choosing near-minimizing mixed
covers gives
$
 G_N\le C_{u,t,m}\Ct_N^{u,0,m}(E)\longrightarrow0.
$
Select $N_k\to\infty$ so that $G_{N_k}\le2^{-2k}$.  Lemma~\ref{lem:moment}
with $\theta=0$, followed by Markov's inequality, gives
\[
 \sum_k\Bigl(
 \gamma_{d,m}(\mathcal B_{N_k})+
 \gamma_{d,m}\{V:W_{N_k}(V)>1\}\Bigr)<\infty.
\]
Thus Borel--Cantelli implies that, for almost every $V$,
$
 \mathcal H_{N_k}^{s,0}(P_VE)\le1
$
for all sufficiently large $k$.  At $\theta=0$ this quantity is exactly
a tail of the macroscopic covering series.  Hence one tail is finite;
adding the finitely many earlier shell contents shows that the whole
series converges.  Therefore $\hDim(P_VE)\le s$.

Let $s\downarrow\hDim E$ through a countable sequence and combine this
upper bound with \eqref{eq:projection-lower} and \eqref{eq:cap-zero} to
obtain
\begin{equation}
 \hDim(P_VE)=\min\{\hDim E,m\}
 \quad\text{for almost every }V\in G(d,m).
 \label{eq:DS-projection}
\end{equation}
For $(d,m)=(2,1)$ this recovers the theorem of Daw and Seuret
\cite[Theorem 2.8]{DS}. Its high-dimensional form here follows from the
preceding proof, not from an unquoted extension of that theorem.

\subsection{One exceptional set for every parameter}
For each rational $\theta\in(0,1]$ and each sign, intersect the
full-measure sets furnished by \eqref{eq:projection-upper} and
\eqref{eq:projection-lower}.  This is a countable intersection and
therefore still has full measure.  On it the projection identities hold
for every positive rational parameter.

Now fix an arbitrary $\theta\in(0,1]$ and choose rational
$\theta_k\to\theta$.  By Proposition~\ref{prop:dimension-continuity},
the two target spectra are continuous at $\theta$, and by
Proposition~\ref{prop:profile-shape} the two profiles are continuous there
as well.  Passing to the limit in the rational identities yields the
same identities at $\theta$.  Finally intersect with the full-measure
set in \eqref{eq:DS-projection} to include $\theta=0$.  The resulting
single exceptional set works simultaneously for every parameter and both
dimensions, proving Theorem~\ref{thm:projection}.

\section{Exact cases and applications}\label{sec:applications}
This section is devoted to criteria for computing the typical
projected dimensions and to applications of the projection theorem
to discrete Moran sets, Lipschitz graphs, and arithmetic sumsets.
For
fixed $m\in \{1,\ldots,d\}$ and $E\subset \R^d$, we write
$g_E^-(\theta)=\lD_\theta^m(E)$ and
$g_E^+(\theta)=\uD_\theta^m(E)$.

\subsection{Macroscopic regularity and a spectrum-only criterion}
\begin{proposition}\label{prop:equality}
Let $E\subset\R^d$ be Borel and fix $m\in \{1,\ldots, d\}$.
\begin{enumerate}
\item If $\hDim E=\Dt^+_{\theta_0}(E)=d_0$, then
$g_E^-(\theta)=g_E^+(\theta)=\min\{d_0,m\}$ for
$0\le\theta\le\theta_0$.
\item If, for some $a\in(0,1)$,
$\Dt^-_\theta(E)=\Dt^+_\theta(E)=d(1-a/\theta)_+$ at all positive
parameters, then $g_E^-(\theta)=g_E^+(\theta)=m(1-a/\theta)_+$.
The zero endpoint is also zero.
\end{enumerate}
\end{proposition}
\begin{proof}
In the first case, for $0\le\theta\le\theta_0$, monotonicity and
Proposition~\ref{prop:basic-endpoints} give
\[
 d_0=\hDim E\le\Dt^-_\theta(E)\le\Dt^+_\theta(E)
 \le\Dt^+_{\theta_0}(E)=d_0.
\]
Thus both source dimensions equal $d_0$.  The upper and lower profile
bounds in Proposition~\ref{prop:recovery} then squeeze both profiles to
$\min\{d_0,m\}$.

In the second case the two source dimensions vanish for
$0<\theta\le a$.  Hence \eqref{eq:source-upper} gives
$g_E^-(\theta)=g_E^+(\theta)=0$ there.  Since both profiles vanish at
$a$, Proposition~\ref{prop:profile-shape} gives, for $\theta>a$,
$
 g_E^\pm(\theta)\le m\left(1-\frac a\theta\right).
$
The $m/d$ lower bound in Proposition~\ref{prop:recovery} gives the
reverse inequality.  At $\theta=0$ use \eqref{eq:cap-zero}.
\end{proof}

Here is a useful sufficient condition for the first criterion with
$\theta_0=1$. Suppose that for some $0<d_0\le d$,
\[
 \limsup_n\frac{\log\max\{1,M_n(E)\}}{n\log2}\le d_0,
\]
and there is a Radon measure $\mu$ concentrated on $E$ such that
\begin{equation}
 \mu(B(x,r))\le Cr^{d_0}\quad(x\in E,r\ge1),\qquad
 \mu(S_n)\ge c2^{d_0n}\quad\text{for infinitely many }n.
 \label{eq:regular-measure}
\end{equation}
The covering condition gives $\Dt^+_1(E)\le d_0$.
Recenter a covering ball at a point of $E$ to use the mass bound.
Every macroscopic cover of one of the heavy shells has normalized
$d_0$-cost at least a fixed positive constant. For smaller exponents
its cost is no lower, so the macroscopic series diverges.
Thus $\hDim E=d_0$, and all typical values equal $\min\{d_0,m\}$.

In particular this applies to coarsely $d_0$-Ahlfors regular sets:
there is a Radon measure $\mu$ concentrated on $E$ such that
$cr^{d_0}\le\mu(B(x,r))\le Cr^{d_0}
$ for any $x\in E$ and $ r\ge1$.
A maximal $3$-separated subset of $E\cap S_n$ gives disjoint unit balls,
each of mass at least $c$. The upper mass bound therefore gives
$O(2^{d_0n})$ points. Radius-$3$ balls about these points cover the shell
portion, and bounded subdivision gives $M_n(E)=O(2^{d_0n})$.
Fixing a point of $E$ and comparing centres also gives
$c_0R^{d_0}\le\mu(B(0,R))\le C_0R^{d_0}$ for all large $R$.
Choose an integer $L$ with $c_0 2^{Ld_0}>2C_0$. Then
\[
 \sum_{h=n+1}^{n+L}\mu(S_h)
 \ge (c_0 2^{Ld_0}-C_0)2^{nd_0}
 \ge \tfrac12c_0 2^{(n+L)d_0}.
\]
At least one of these $L$ shells satisfies
$\mu(S_h)\ge(c_0/(2L))2^{hd_0}$, yielding infinitely many shells in
\eqref{eq:regular-measure}. No control below the unit scale is needed.

\subsection{High-dimensional discrete digit Moran sets}\label{sec:moran}
We now apply the first criterion to a vector-digit version of discrete
Moran constructions. The fixed-base, variable-base, and discrete gasket
settings were studied in \cite{YDX,LCX,DYX}. The definition below is the
bounded-base higher-dimensional digit model used here; it does not assert
that all discrete Moran constructions in those papers have this form.

\begin{definition}\label{def:moran}
Fix $d\ge2$, integers $2\le b_j\le b^*<\infty$, and digit sets
$\mathcal A_j\subset\{0,\ldots,b_j-1\}^d$ satisfying
\emph{$n_j:=\#\mathcal A_j\ge2$ at every level}. Put
\[
 B_0=Q_0=1,\qquad B_k=\prod_{j=1}^k b_j,\qquad
 Q_k=\prod_{j=1}^k n_j
 \]
 and $ E=\bigcup_{k\ge1}E_k\subset\Z^d,$ where
 \[
 E_k=\left\{\sum_{j=1}^k B_{j-1}a_j:
                        a_j\in\mathcal A_j\right\}.
\]
We do not require $0\in\mathcal A_j$, so the levels need not be nested.
\end{definition}

\begin{proposition}\label{prop:moran-endpoints}
For the set in Definition~\ref{def:moran},
\begin{equation}
 \hDim E=\UMdim E=s, \quad
\text{where $s=\limsup_{k\to\infty}
       \frac{\sum_{j=1}^k\log n_j}{\sum_{j=1}^k\log b_j}.$}
 \label{eq:moran-dimension}
\end{equation}
\end{proposition}
\begin{proof}
Uniqueness of vector digit expansions gives $\#E_k=Q_k$.
If $\ell>k$ and a point of $E_\ell$ lies in $[0,B_k)^d$, all its digits
above level $k$ must be zero, so that point belongs to $E_k$.
Since every $n_j\ge2$,
\[
 Q_k\le\#(E\cap[0,B_k)^d)\le\sum_{\ell=1}^k Q_\ell<2Q_k.
\]
For an arbitrary large radius $r$, choose $k$ with $B_k\le r<B_{k+1}$. The upper cube estimate at $k+1$ gives
$\#(E\cap B(0,r))\le2Q_{k+1}$. Conversely, $E_k\subset B(0,2\sqrt d\,B_k)$ gives a count at least $Q_k$ at that radius. Since $b_j\le b^*$ and $n_j\le(b^*)^d$, changing $k$ by one changes both logarithms by bounded amounts. These two bounds imply $\UMdim E=s$. Lattice point counts and unit covering numbers are comparable, so Proposition~\ref{prop:mass-endpoint} also gives $\hDim E\le s$.

Fix $0<u<s$. The sequence $Q_k/B_k^u$ is unbounded, so there are
infinitely many record indices $k$ with
$Q_k/B_k^u\ge Q_\ell/B_\ell^u$ for $0\le\ell\le k$.
Let $\mu_k$ be uniform probability on $E_k$.
For $1\le r<B_k$, choose $\ell<k$ with $B_\ell\le r<B_{\ell+1}$.
A ball of radius $r$ meets at most $C_{d,b^*}$ grid cubes of side
$B_\ell$, and each contains at most $Q_\ell$ points of $E_k$.
The record property therefore gives
\begin{equation}
 \mu_k(B(x,r))\le C_{d,b^*}\frac{Q_\ell}{Q_k}
 \le C_{d,b^*}\left(\frac{B_\ell}{B_k}\right)^u
 \le C_{d,b^*}\left(\frac{r}{B_k}\right)^u.
 \label{eq:moran-record}
\end{equation}
For $r\ge B_k$ the same bound follows from total mass one.
Choose a fixed $\rho\in(0,1)$ so small that $C_{d,b^*}\rho^u<1/2$.
At every sufficiently large record index, $\rho B_k\ge1$, and $\mu_k$
gives mass at least $1/2$ to
$\mathcal R_k=\{x:\rho B_k\le|x|\le\sqrt d\,B_k\}$.
Let $J_k$ be the set of dyadic shell indices meeting $\mathcal R_k$.
Then $\#J_k$ is bounded independently of $k$, and $2^n\asymp B_k$
uniformly for $n\in J_k$. Take arbitrary legal macroscopic covers
$\{B(x_{n,i},r_{n,i})\}_i$ of $E\cap S_n$ for these indices.
Since $\mu_k$ is concentrated on $E_k\subset E$, \eqref{eq:moran-record}
yields
\[
 \frac12\le\mu_k(\mathcal R_k)
 \le C\sum_{n\in J_k}\sum_i\left(\frac{r_{n,i}}{B_k}\right)^u
 \le C'\sum_{n\in J_k}\sum_i\left(\frac{r_{n,i}}{2^n}\right)^u.
\]
The constants may depend on $d,b^*,u,\rho$, but not on $k$ or the covers.
Taking the covering infimum separately over the finitely many shells gives
$
 \sum_{n\in J_k}\mathcal{H}_n^u(E)\ge(2C')^{-1}>0.
$
As $B_k\to\infty$, a subsequence of record indices can be chosen with
pairwise disjoint groups $J_k$. Summing over these groups forces the
macroscopic covering series to diverge at $u$. Hence $\hDim E\ge u$;
letting $u\uparrow s$ proves the assertion.
\end{proof}

\begin{corollary}
\label{cor:moran-projection}
Let $E$ be as in Definition~\ref{def:moran}, with $s$ given by
\eqref{eq:moran-dimension}. For every $1\le m<d$, there is a full-measure
set $\mathcal G_{E,m}\subset G(d,m)$ such that
\begin{equation}
 \Dt^-_\theta(P_VE)=\Dt^+_\theta(P_VE)=\min\{m,s\}
 \label{eq:moran-projection}
\end{equation}
for all $V\in\mathcal G_{E,m}$ and  $0\le\theta\le1$.
In the discrete self-similar case $b_j=b$ and
$\mathcal A_j=\mathcal A\ni0$, the value is
$\min\{m,\log\#\mathcal A/\log b\}$.
\end{corollary}
\begin{proof}
Since $E\subset\mathbb Z^d$ is uniformly discrete,
Propositions~\ref{prop:moran-endpoints} and~\ref{prop:mass-endpoint} give
$\Dt^+_1(E)=\UMdim E=s$.  On the common full-measure set in
Theorem~\ref{thm:projection}, the zero endpoint gives
$\hDim(P_VE)=\min\{m,s\}$.  For every $\theta\in[0,1]$,
Theorem~\ref{thm:projection}, \eqref{eq:source-upper}, and monotonicity of
the source spectrum yield
\[
 \Dt^+_\theta(P_VE)=\uD_\theta^m(E)
 \le\min\{m,\Dt^+_\theta(E)\}
 \le\min\{m,\Dt^+_1(E)\}=\min\{m,s\}.
\]
Together with Proposition~\ref{prop:basic-endpoints},
\[
 \min\{m,s\}=\hDim(P_VE)
 \le\Dt^-_\theta(P_VE)\le\Dt^+_\theta(P_VE),
\]
so all inequalities are equalities.  This proves
\eqref{eq:moran-projection} simultaneously for all $\theta$, without any
further exceptional set.
\end{proof}

\begin{example}
\label{ex:moran-exception}
The conclusion of Corollary~\ref{cor:moran-projection} cannot in general
hold for every $V$. Take $b_j=3$ and $\mathcal A_j=\{0,2\}^d$.
Then $s=d\log2/\log3$. A coordinate $m$-plane has projection equal to
the corresponding $m$-dimensional product digit set, and hence
\[
 \Dt^-_\theta(P_VE)=\Dt^+_\theta(P_VE)
 =\frac{m\log2}{\log3}
 <\min\left\{m,\frac{d\log2}{\log3}\right\}
 \quad(m<d,\ 0\le\theta\le1).
\]
The endpoint argument above also applies in dimension one.
For a specified subspace $V$ for which $P_VE$ is locally finite, the deterministic implication
\[
 \hDim(P_VE)=\UMdim(P_VE)=t_V
 \quad\Longrightarrow\quad
 \Dt^-_\theta(P_VE)=\Dt^+_\theta(P_VE)=t_V
 \quad(0\le\theta\le1)
\]
always holds. It gives a fixed-subspace statement when its endpoint
hypothesis can be checked, but it need not give the typical value.
\end{example}

\subsection{Lipschitz graphs}
We first isolate the covering principle needed for the graph application.

\begin{lemma}\label{lem:relative-cover-profile}
Let $E\subset\R^d$ be Borel and let $1\le m\le d$. Suppose there is a
constant $C$ such that, for every $n$ with $A_n:=E\cap S_n\ne\varnothing$,
every $x\in A_n$, and every $1\le r\le\rho$,
\begin{equation}
 A_n\cap B(x,\rho)\text{ can be covered by at most }
 C(\rho/r)^m\text{ balls of radius }r
 \label{eq:relative-m-cover}
\end{equation}
with centres in $A_n$. Then, for every $0<\theta\le1$,
\[
 \lD_\theta^m(E)=\Dt^-_\theta(E),\quad
 \uD_\theta^m(E)=\Dt^+_\theta(E).
\]
\end{lemma}
\begin{proof}
The upper bounds for the profiles are already contained in
\eqref{eq:source-upper}. The hypothesis with $r=1$ also gives
$\Dt^+_\theta(E)\le m$. Indeed, if $A_n\ne\varnothing$, choose
$x\in A_n$. Since $A_n\subset B(x,2^{n+1})$, \eqref{eq:relative-m-cover}
with $\rho=2^{n+1}$ covers $A_n$ by $O(2^{mn})$ unit balls. Thus it
remains to prove the reverse inequalities.

Fix $0<u<v<m$, $n\ge N$, and set $R=2^{n-N\theta}$. Start with a mixed
cover in the definition of $\Mt_R^{v,m}(A_n)$. A ball of radius at most
$R$ is retained. If a covering ball $B(x,r)$ has radius $r>R$, then its
centre lies in $A_n$, and \eqref{eq:relative-m-cover} covers
$A_n\cap B(x,r)$ by at most $C(r/R)^m$ balls of radius $R$, all centred
in $A_n$. Replacing every large ball in this way gives an admissible
cover for $\mathcal H_{n,N}^{v,\theta}(E)$ and hence
$
 \mathcal H_{n,N}^{v,\theta}(E)
 \le C2^{-nv}\Mt_R^{v,m}(A_n).
$
Lemma~\ref{lem:mixed} therefore yields
\begin{equation}
 \mathcal H_N^{v,\theta}(E)
 \le C_{u,v,m}(N+1)2^{-N\theta(v-u)}
       \Ct_N^{u,\theta,m}(E).
 \label{eq:relative-profile-comparison}
\end{equation}

Suppose, for one of the two signs, that the corresponding profile is
strictly smaller than the corresponding macroscopic intermediate dimension. Since both are
at most $m$, we may choose $v<m$ such that
$
 \Dpm_\theta^{\pm,m}(E)<u<v<\Dt_\theta^\pm(E),
$
with the same sign throughout. At exponent $u$ the capacity tail has
nonpositive normalized growth, along a subsequence for the lower sign
and at all sufficiently large scales for the upper sign. The factor
$2^{-N\theta(v-u)}$ in \eqref{eq:relative-profile-comparison} then makes
the covering tail have negative normalized growth at exponent $v$,
contradicting $v<\Dt_\theta^\pm(E)$. Hence equality holds for both signs.
\end{proof}

The following result shows that almost every orthogonal projection onto an $m$-dimensional subspace preserves both macroscopic intermediate dimensions of unbounded Borel subsets of $k$-dimensional Lipschitz graphs, simultaneously for all $\theta\in[0,1]$, whenever $k\le m<d$.
\begin{theorem}\label{thm:graphs}
Let $1\le k\le m<d$, let $A\subset\R^k$ be unbounded and Borel,
and let $f:\R^k\to\R^{d-k}$ be $L$-Lipschitz. Put
$E_f=\{(x,f(x)):x\in A\}$. Then
\[
 \Dt^\pm_\theta(E_f)=\Dt^\pm_\theta(A),\qquad
 g_{E_f}^\pm(\theta)=\Dt^\pm_\theta(A)
 \quad(0\le\theta\le1).
\]
Thus both dimensions are preserved by almost every projection onto an
$m$-plane, simultaneously at all parameters, even when the source spectra
differ. For an individual $V\in G(d,m)$, put
$T_Vx=P_V(x,0)$ and $U_Vy=P_V(0,y)$. Preservation also holds for that $V$
whenever
\begin{equation}
\sigma_{\min}(T_V)>L\|U_V\|,
 \label{eq:graph-individual}
\end{equation}
where $\sigma_{\min}(T_V):=\inf_{|x|=1}|T_Vx|$  denotes the smallest singular value of $T_V$.
\end{theorem}
\begin{proof}
The graph map $x\mapsto(x,f(x))$ is bi-Lipschitz. More generally, the
same covering estimate holds for every subset of the graph: if a graph
portion lies in a ball of radius $\rho$, then its first-coordinate
projection lies in a $k$-dimensional ball of radius $\rho$. Cover that
projection by $O_{k,L}((\rho/r)^k)$ balls of radius comparable to
$r/\sqrt{1+L^2}$, choosing the centres in the projected subset, and lift
them to the graph. Applied to each shell $E_f\cap S_n$, and using
$k\le m$, this gives \eqref{eq:relative-m-cover}.  Lemma~\ref{lem:relative-cover-profile}
and Proposition~\ref{prop:qi} therefore imply, for $\theta>0$,
$
 g_{E_f}^\pm(\theta)=\Dt_\theta^\pm(E_f)=\Dt_\theta^\pm(A).
$
At $\theta=0$, Proposition~\ref{prop:recovery} gives
$g_{E_f}^\pm(0)=\hDim E_f$, because $\hDim E_f\le k\le m$; the
quasi-isometry invariance and Proposition~\ref{prop:basic-endpoints} then
give $\hDim E_f=\hDim A=\Dt_0^\pm(A)$.  The almost-everywhere
projection assertion follows from Theorem~\ref{thm:projection}.
For a fixed subspace satisfying \eqref{eq:graph-individual},
\[
 |T_V(x-y)+U_V(f(x)-f(y))|
 \ge(\sigma_{\min}(T_V)-L\|U_V\|)|x-y|.
\]
The reverse Lipschitz estimate is immediate, so the projected graph map
is bi-Lipschitz. In the planar case the condition reduces to
$|\cos\alpha|>L|\sin\alpha|$.
\end{proof}

\subsection{Arithmetic sumsets}
In this subsection $(d,m)=(2,1)$, and $g_E^\pm$ denotes the line-projection profile.
\begin{theorem}\label{thm:sumsets}
Let $A,B\subset\R$ be unbounded uniformly discrete sets. Suppose
$0<a,b\le1$ and
\[
 \begin{aligned}
 \#(A\cap B(x,r))&\le Cr^a,&\#(B\cap B(y,r))&\le Cr^b
       &&(x,y\in\R,r\ge1),\\
 \#(A\cap[-R,R])&\asymp R^a,&\#(B\cap[-R,R])&\asymp R^b
       &&(R\to\infty).
 \end{aligned}
\]
Then, for almost every $t\in\R$, simultaneously for $\theta\in[0,1]$,
\[
 \Dt^-_\theta(A+tB)=\Dt^+_\theta(A+tB)=\min\{a+b,1\}.
\]
The sumset itself need not be locally finite.
\end{theorem}
\begin{proof}
Counting measure on $A\times B$ has upper ball bound $Cr^{a+b}$
and centre-ball growth comparable to $R^{a+b}$. Comparing radii that
differ by a sufficiently large fixed factor gives infinitely many
heavy dyadic shells. Unit covering numbers have upper rate at most
$a+b$. The criterion \eqref{eq:regular-measure} and
Proposition~\ref{prop:equality} give typical projected value
$\min\{a+b,1\}$ for both signs and all parameters.
For $\cos\alpha\ne0$,
$P_\alpha(A\times B)=\cos\alpha\,(A+\tan\alpha\,B)$.
Dilation invariance and the locally bi-Lipschitz change of variable
$t=\tan\alpha$ transfer the null set to slopes.
\end{proof}

For nonconstant polynomials $p,q\in\Z[x]$ of degrees $r,s$,
Theorem~\ref{thm:sumsets} yields
\[
 \Dt^\pm_\theta\bigl(p(\Z)+tq(\Z)\bigr)=\min\{1,1/r+1/s\}
\]
for almost every $t$, simultaneously at all parameters.
Indeed, $p(\Z)$ has $\asymp R^{1/r}$ distinct values in $[-R,R]$.
On each sufficiently distant half-line the polynomial is monotone and
$|p(u)-p(v)|\gtrsim|u-v|^r$ for integers $u,v$ there.
Thus an interval of length $2\rho\ge2$ contains at most $C_p\rho^{1/r}$
values; the finite remaining part changes only the constant. Apply the
same argument to $q$. Related arithmetic counting problems were studied
in \cite{LM,Glasscock}.

For an integer base $b\ge2$ and a digit set
$0\in\mathcal D\subset\{0,\ldots,b-1\}$ with $M=\#\mathcal D\ge2$,
let
\[
 \Lambda(b,\mathcal D)=
 \left\{\sum_{j=0}^{k-1}d_jb^j:k\ge1,\ d_j\in\mathcal D\right\}.
\]
Its counting measure is coarsely $d_0$-Ahlfors regular, where
$d_0=\log M/\log b$. To prove this, choose $b^k\le r<b^{k+1}$.
Keeping the higher digits of a point and varying its $k$ lower digits
gives $M^k$ points within radius $r$. Conversely, an interval of length
$2r$ meets at most $2b+2$ grid intervals of length $b^k$, and each such
interval contains at most $M^k$ points. The preceding exact formulae
therefore apply to these digit sets, including pairs with different bases.

In particular, let $\Lambda=\Lambda(4,\{0,1\})$. This set is a
spectrum of the equal-weight self-similar measure with maps $x/4$ and
$(x+2)/4$, see \cite{JP}. Theorem~\ref{thm:sumsets} gives typical sumset
dimension one, for both signs and all parameters. The special slopes
satisfy
\[
 \Dt^\pm_\theta(\Lambda)=\tfrac12,\qquad
 \Dt^\pm_\theta(\Lambda+\Lambda)=\frac{\log3}{\log4},\qquad
 \Dt^\pm_\theta(\Lambda+2\Lambda)=1.
\]
The last two identities follow from digit addition without carries:
the sums are $\Lambda(4,\{0,1,2\})$ and $\{0,1,2,\ldots\}$,
respectively. These statements concern the geometry of the frequency
sets and do not assert that the sumsets are spectra. Bounded perturbations
of either factor preserve all these dimensions, by
Proposition~\ref{prop:qi} applied at each fixed slope.

\section*{Statements and Declarations}
\subsection*{Funding}
This work was supported by the National Natural Science Foundation of China
(Grant Nos. 12171107 and 12571096).

\subsection*{Data Availability}
This article is theoretical. No research datasets were generated or analysed.

\subsection*{Competing Interests}
The authors declare that they have no financial or non-financial interests
relevant to the content of this article.

\subsection*{Use of generative AI}
During preparation of this manuscript, the authors used Chat
GPT (OpenAI) to assist with language revision. The authors are responsible for checking the mathematical arguments, citations, and final text.

\end{document}